\documentclass[14pt]{article}
\usepackage{amsmath,amsfonts,amssymb,amsthm,enumerate}
\usepackage{color}
\def\blue{\textcolor{\blue}}

\renewcommand{\leq}{\leqslant}

\renewcommand{\le}{\leqslant}
\renewcommand{\ge}{\geqslant}

\newcommand{\R}{ { \mathbb{R} } }
\newcommand{\C}{ { \mathbb{C} } }

\newcommand{\N}{ { \mathbb{N} } }

\newcommand{\Z}{ { \mathbb{Z} } }

\newcommand{\Ldiv}{ L_{\rm div} }
\newcommand{\eps}{ \varepsilon }
\newcommand{\ep}{ \varepsilon }

\def\curl{\operatorname{curl}}
\def\Div{\operatorname{div}}

\newcommand{\dt}{ \partial_t }

\newcommand{\tendlorsque}[2]{\mathop{\longrightarrow}\limits_{#1\rightarrow#2}}

\newtheorem{Theorem}{Theorem}
\newtheorem{Definition}[Theorem]{Definition}

\newtheorem{Proposition}[Theorem]{Proposition}
\newtheorem{Lemma}[Theorem]{Lemma}
\newtheorem{Remark}{Remark}

\begin{document}

\title{ Cauchy problem and quasi-stationary limit for 
the Maxwell-Landau-Lifschitz and Maxwell-Bloch equations}
\author{Eric Dumas\footnote{Universit\'e Grenoble 1 - 
Institut Fourier - 100, rue des math\'ematiques - 
BP 74 - 38402 Saint Martin d'H\`eres FRANCE }, 
Franck Sueur\footnote{Laboratoire Jacques-Louis Lions -
Universit\'e Pierre et Marie Curie - Paris 6 -
75252 Paris  FRANCE }}
\maketitle

\begin{abstract}
In this paper we continue the investigation of  
the Maxwell-Landau-Lifschitz and Maxwell-Bloch equations. 
In particular we extend some previous results about the Cauchy 
problem and the quasi-stationary limit to the case  where the 
magnetic permeability  and the electric permittivity are variable.
\end{abstract}



\tableofcontents


\section{Introduction}

\paragraph{The models.} 
This paper deals with two physical models which describe the 
propagation of electromagnetic waves, that is of the magnetic 
field $H$ and of the electric field $E$, in some special medium 
which occupies an open subset $\Omega$ of $\R^3$, with magnetic 
permeability  $\mu$ and electric permittivity $\ep$. 
In both cases we denote by $\overline{f}$ the extension of a 
function $f$ by $0$ outside the set $\Omega$. The time variable 
is $t\ge0$, and the space variable is $x\in\R^3$. 

The first model refers to {\bf Maxwell-Landau-Lifschitz equations} 
(see \cite{Brown63} and \cite{LL69} for Physics references). 
The magnetic field $H$ and the electric field $E$  satisfy the 
Maxwell  equations in  $\R^3$:
\begin{equation} \label{maxwellLL} 
\left\{
\begin{split}
& \mu \dt H + \curl E = -  \mu \dt \overline{M} , \\
& \ep \dt E - \curl H =0 , \\
&\Div   \mu ( H +  \overline{M} ) =0 , \\ 
&\Div  \ep E=0 ,
\end{split} \right.
\end{equation}
where $M$ stands for the magnetic moment in the ferromagnet  
$\Omega$ and takes  values in the unit sphere of $\R^3$. 
It is solution to the Landau-Lifschitz equation: 
\begin{equation}
\label{LL}
\dt {M} =  \gamma M \wedge H_T - \alpha M \wedge ( M \wedge H_T)  
\quad \text{ for }  x \in  \Omega,
\end{equation}
where $\gamma\neq0$ is the gyromagnetic constant, and $\alpha>0$ is 
some damping coefficient. Neglecting the exchange phenomenon, 
the total magnetic field $H_T$ is the sum 
\begin{equation}
\label{eq:Htot}
H_T = H + H_a(\overline M) + H_{\rm ext},
\end{equation}
where the anisotropy field writes $H_a(\overline M)=
\nabla_M\Phi(\overline M)$, for some convex function $\Phi$, 
and $H_{\rm ext}$ is some applied (exterior) magnetic field. 

The second model refers to  {\bf Maxwell-Bloch equations} (see 
for example \cite{Bohm79}, \cite{Boyd92}, \cite{CTDRG88}, 
\cite{NM92}, \cite{PP69}, \cite{SSL77}). 
In this setting $\Omega$ denotes some quantum medium with 
$N \in \N $ energy levels described by a Hermitian, non-negative, 
$N \times N$  density matrix $\rho$. 
Assuming the usual dipolar approximation, these quantum states change under the action 
of an electric field $E$ by the quantum Liouville-Von Neumann (or Bloch) equation:
\begin{equation}
\label{MB}
i \dt \rho = [\Lambda-E\cdot\Gamma,\rho ] + i Q(\rho).
\end{equation}
The $N \times N$ Hermitian symmetric matrix $\Lambda$, with entries 
in $\C$, represents the (electromagnetic field-) free Hamiltonian of 
the medium. The dipole moment operator $\Gamma$ 
is a $N \times N$ Hermitian matrix, with entries in $\C^3$, 
and depends on the material considered. The (linear) relaxation 
term $Q(\rho)$ takes dissipative effects into account (see 
\cite{BBR01}, \cite{BB06}, \cite{Loudon91}).  
The polarization $P$ of the matter is given by the constitutive law
$P = \mathrm{ Tr} (\Gamma \rho)$ which influences back the electric 
field $E$. Again, the electromagnetic field satisfies the Maxwell  
equations in  $\R^3$:
\begin{equation} \label{maxwellMB} 
\left\{
\begin{split}
& \mu \dt H + \curl E = 0 , \\
& \ep \dt E - \curl H = -\dt \overline{P} ,  \\
& \Div (\ep  E  + \overline{P} ) = 0,  \\
& \Div  \mu H = 0 .
\end{split} \right.
\end{equation}

\paragraph{Cauchy problems.} 
We first address the questions of global existence, 
uniqueness and stability for the Cauchy problem associated 
with these equations. The physically relevant solutions 
have finite energy: they satisfy the usual ($L^2$) energy estimates. 
Mathematically, this regularity leads to weak solutions and 
is usually not enough to ensure the desired uniqueness 
and stability properties (requiring for these hyperbolic 
semilinear systems in space dimension 3, in the general theory, 
$H^s$ Sobolev regularity with $s>3/2$). 

However, in the case of the Maxwell-Landau-Lifschitz system, 
Joly, M\'etivier and Rauch \cite{jmr} noticed 
that specific (algebraic) properties of the nonlinearities, 
as well as (geometric) properties of the differential operator 
involved, allowed to show the existence of global finite energy 
solutions (essentially, using compensated compactness arguments) 
enjoying stability properties. Furthermore, only a small amount 
of regularity ($\curl H$ and $\curl E$ in $L^2$) ensures 
uniqueness. This is achieved using dispersive properties of 
the system; namely, a limit Strichartz estimate controlling the 
$L^2_tL^\infty_x$ norm of (a limited frequency part of) the fields 
$H$ and $E$. These results were obtained for equations posed in 
the whole space ($\Omega=\R^3$) and for constant coefficients 
$\varepsilon$ and $\mu$. 

In practice, the various coefficients of the system may not be 
constant. Typically, the magnetic permeability and electric 
permittivity may depend on the space variable $x$ and have jumps 
across the boundary of the domain $\Omega$. 
 
Adapting the above mentioned compensated compactness argument, 
Jochmann established in \cite{jochmann} the existence and weak 
stability of global finite energy solutions for the 
Maxwell-Landau-Lifschitz system, considering any domain 
$\Omega\subset\R^3$, and variable, possibly discontinuous coefficients 
($\varepsilon, \mu \in L^\infty(\R^3)$). In the (space) 2-dimensional 
case, we refer to the work of Haddar \cite{Had00}.

Concerning the Maxwell-Bloch system, the first author noticed that 
it shares with the Maxwell-Landau-Lifschitz some of its structural 
properties. This author thus showed in \cite{dumas} results on 
existence and uniqueness of global finite energy solutions, similar 
to the ones of Joly, M\'etivier and Rauch, but for some general class 
of systems including the two models above. Again, these results 
where obtained for equations posed in the whole space 
and for constant coefficients $\varepsilon$ and $\mu$. 

Here, we continue this study, again for a general class of systems 
including the Maxwell-Landau-Lifschitz equations and the Maxwell-Bloch 
equations, so as to enlight the similarities and differences between 
these two models. Adapting Jochmann's method, we show the existence 
and stability of global finite energy solutions, for a given domain 
$\Omega\subset\R^3$, and $L^\infty$ coefficients. 
Then, for smooth coefficients, constant out of some compact set, 
we prove a limit Strichartz estimate analogous to the one obtained 
by Joly, M\'etivier and Rauch in the constant coefficient case. 
This allows us to show propagation of regularity and uniqueness 
when initially, $\curl H$ and $\curl E$ belong to $L^2(\R^3)$. 
As a corollary of a result of Saint-Raymond \cite{LSR}, we also 
infer generic uniqueness of the global finite energy solutions. 

\paragraph{Quasi-stationary limits.} 
Next, we turn to the problem of the so-called quasi-stationary limit. 
Physically, this regime appears when the domain $\Omega$ is small 
compared to the wavelength. Mathematically, it amounts to some 
long-time asymptotics (replacing in the equations $\dt$ by $\eta\dt$, 
for some small parameter $\eta$) with weak nonlinearities (also 
scaled so as to have an amplitude of size $\eta$). 

Jochmann showed in \cite{jochmann} the weak convergence of the 
corresponding solutions to the Maxwell-Landau-Lifschitz system 
towards the solutions of some reduced system driven by 
the magnetization, using the weak stability property. 
Starynkevitch extended this result, proving strong and 
global-in-time convergence in the constant coefficient 
case in \cite{starynkevitch}, thanks to local energy estimates 
performed on the explicit fundamental solution of the associated 
wave equation. He also obtained the same result in the case of 
smooth coefficients, constant out of some compact set, in 
\cite{starynkevitchPhD}, thanks to dispersive estimates obtained 
from resolvent estimates on elliptic operators. 

Here, we apply the same methods to our general systems to get 
weak and strong convergence in the quasi-stationary limit. 
For the latter however,  some time integrability assumption 
is needed to conclude, which is satisfied by 
the Maxwell-Landau-Lifschitz system 
(since $\dt M \in L^2((0,\infty)\times\Omega)$), but we do not 
know if the Maxwell-Bloch system enjoys such a property.

\begin{Remark}
Taking exchange energy into account, one should add to the total 
magnetic field in \eqref{eq:Htot} a term $-K\Delta M$. The resulting 
system is then parabolic. We refer to \cite{AB09}, \cite{AS92}, 
\cite{CF98}, \cite{CF01a}, \cite{CF01b} and \cite{Visintin85} for works 
on the (weak or strong) Cauchy problem, and long-time asymptotics.  
\end{Remark}


\section{Main results}

Let us stress that we do not assume that  $\Omega $ is bounded,  
for the moment. 
To deal with both the Maxwell-Landau-Lifschitz system (\ref{maxwellLL})-(\ref{LL}) and the Maxwell-Bloch system 
(\ref{MB})-(\ref{maxwellMB}), we put these two models above into  
a single class of systems consisting in the coupling of the 
Maxwell equations (with the fields $H$ and $E$ as unknowns) 
with some ODE (corresponding to a third unknown variable). 
The resulting sytem is symmetrizable hyperbolic, 
with semilinear nonlinearity, and some structure 
assumptions are made, such as affine dependence of the nonlinearity 
with respect to the electromagnetic field, and \emph{a priori}  
pointwise estimates on the third unknown variable. 
One of the key points in our study is that the electromagnetic 
fields decompose into an ``irrotational'' part, which is directly 
related to this third unknown, and a ``divergence free'' part, 
which solves some wave equation.

\subsection{An abstract setting}
\label{abstract}

On any finite-dimensional vector space $\R^N$, we denote by 
$u \cdot u'$ the usual scalar product between vectors $u$ and $u'$, 
and by $|\cdot|$ the associated norm. For all $r>0$, $B_r$ denotes 
the (closed) ball centered at $0$, with radius $r$.

We consider two scalar functions $\kappa_1 (x)$ and $ \kappa_2  (x)$, which are uniformly positive:
\begin{equation} \label{hyp:kappa}
\text{for } i=1,2, \quad \kappa_i \in L^\infty (\R^3), \quad 
\text{and} \quad \exists c>0, \kappa_i \geqslant c .
\end{equation}
We denote by $H_{\curl}$ the space of functions 
$f$ in $L^2 (\R^3,\R^3)$ with $\curl f $ in  $L^2(\R^3,\R^3)$.
We consider the operator $B$ defined by 
\begin{eqnarray*} 
B (u_1 ,u_2 ) = (\kappa_1^{-1}  \curl u_2 ,-\kappa_2^{-1} \curl u_1 ) 
\quad \text{for}   
\quad u:=(u_1 ,u_2 ) \in D(B) := H_{\curl} \times  H_{\curl} .
\end{eqnarray*}
This is a skew self-adjoint operator on the Hilbert space $L^2 (\R^3 , \R^6)$ endowed with the 
scalar product 
\begin{eqnarray*} 
\langle (u_1 ,u_2) , (u'_1 ,u'_2) \rangle_{\kappa_1 , \kappa_2 } := 
\int_{\R^3} (\kappa_1 u_1 \cdot u'_1 + \kappa_2 u_2 \cdot u'_2 ) dx .
\end{eqnarray*}
We denote by $P (u_1 ,u_2) := \big(P_1 u_1 , P_2 u_2 \big)$ 
the orthogonal projector on $(\ker B)^\perp$ with respect to 
the weighted scalar product above, so that for $i=1,2$,
\begin{equation} \label{eq:ranP,kerP}
\text{ran } P_i 
= \{ u_i  \in L^2 (\R^3 ,  \R^3) \mid \Div(\kappa_i u_i)=0 \}, 
\quad 
\text{ran } (Id - P_i ) 
=  \{ u_i  \in L^2 (\R^3 ,  \R^3) \mid \curl ( u_i)=0 \} .
\end{equation}
We consider a function 
$F : \R^3\times\R^d\times\R^6\rightarrow\R^d$, where $d\in\N$, 
affine in its third variable, and  written 
\begin{eqnarray} 
\label{affine} 
F ( x,v,u) = F_0 (x, v) + F_1 (x, v) u .
\end{eqnarray}

For each $j=0,1$, $F_j$ is measurable with respect to $x$ 
and continuously differentiable with respect to $v$. 
Furthermore, 
\begin{equation} \label{maj1} 
\begin{split}
\text{for } j=0,1, \quad 
& \text{for almost all } x\in\R^3, \, F_j(x,0)=0, \\
\text{and} \quad
& \forall R>0, \text{ for almost all } x\in\R^3, \, \forall v\in B_R, 
\quad |F_j (x,v)| + | \partial_v  F_j (x,v)| \leqslant C_F (R) .
\end{split}
\end{equation}
Finally, we assume that there exists $K \ge 0$ such that:
\begin{eqnarray}
\label{maj2} 
\text{for almost all } x\in\R^3, \forall (v,u) \in \R^d\times\R^6, 
\quad F(x,v,u) \cdot v   \leqslant   K |v|^2 .   
\end{eqnarray}
\begin{Remark} 
The constant $K$ above may sometimes be taken equal to zero. 
In this case, Estimate \eqref{estimweak1} in Theorem \ref{th:weak} 
is improved, since $v$ does not undergo any growth. 
This is the case for the Maxwell-Landau-Lifschitz model, 
as well as for the Maxwell-Bloch model, when only transverse 
relaxation is taken into account ($Q(\rho)=-\gamma\rho_{od}$, 
for some $\gamma\ge0$, and with $\rho_{od}$ the off-diagonal 
part of $\rho$). 
\end{Remark}
We also consider a function
 $l=(l_1,l_2) \in (L^\infty (\R^3 , L(\R^d ,\R^3) ))^2$, 
where $L(\R^d ,\R^3 )$ denotes the space of linear functions from $\R^d$ to $\R^3$.
We introduce the following shorthand notation: for any $x  \in \R^3$, $(\kappa^{-1} \cdot l )(x)$ 
is the mapping from $\R^d$  to $\R^6$, such that
$$
\text{for almost all } x\in\R^3, \, \forall v \in \R^d, \quad 
(\kappa^{-1} \cdot l)(x) v := 
(\kappa_1 (x)^{-1} l_1 (x) v,\kappa_2 (x)^{-1}l_2 (x) v).
$$ 
Then, for any $U := (u,v)$ in 
\begin{equation*} 
{\bf L}^2 :=  L^2  (\R^3 ,  \R^6 ) \times L^2 ( \Omega, \R^d ),
\end{equation*}
the conditions 
\begin{equation*} 
\Div(\kappa_1 u_1 - l_1 \overline{v})=0 , \quad
\Div(\kappa_2 u_2 - l_2 \overline{v})=0 ,
\end{equation*}
may be equivalently written 
\begin{eqnarray} 
\label{divgen1}  
(Id - P) (u -  (\kappa^{-1} \cdot l) \overline{v} ) =0 .
\end{eqnarray}
We look  for $ U  \in C([0,\infty),  \bf{L}^2 )$, with  
\begin{eqnarray} 
\label{nunu}  
v \in L^\infty_{loc}((0,\infty), L^\infty(\Omega,\R^d)), 
\end{eqnarray}
solution to  
\begin{eqnarray}
\label{MBgen} 
( \dt + B ) u &=&  (\kappa^{-1} \cdot l)  F(x,\overline{v},u)
\quad \text{ for } x  \in  \R^3 , \\ 
\label{MBgen2} 
\dt v &=& F(x,v,u) \quad \text{ for } x  \in \Omega ,
\end{eqnarray}
and \eqref{divgen1}. 
Here, the solution is understood in the distributional sense, 
noticing that \eqref{nunu} gives sense to the nonlinear term, 
since the function $F(x,v,u)$ is affine in $u$.
\begin{Remark}
Equations \eqref{divgen1}-\eqref{MBgen}-\eqref{MBgen2} reduce to 
the Maxwell-Landau-Lifschitz system (\ref{maxwellLL})-(\ref{LL}) 
when $u_1 = H$, $u_2 = E$, $v=M$ (with $d=3$), $\kappa_1 = \mu$, 
$\kappa_2 =  \eps$, $l_1 = -\mu$, $l_2 = 0$,  $F(x,v,u) = \gamma 
v \wedge (u_1 + H_a(v) + H_{\rm ext}) - 
\alpha v \wedge ( v \wedge (u_1 + H_a(v) + H_{\rm ext}) )  $
and to the Maxwell Bloch  system (\ref{MB})-(\ref{maxwellMB})  
when $u_1 = H$, $u_2 = E$,  $v= \rho$ (with $d=N^2$), $\kappa_1 =   \mu$, $\kappa_2 =  \eps$, $l_1 = 0$, $l_2 = \mathrm{Tr} (\Gamma\cdot)$, $F(x,v,u) = - i [\Lambda- u_2 \cdot \Gamma, v ] +  Q(v)$. 
The exterior magnetic field above is usually depending on time. 
We did not consider such time-dependent coefficients in our study, 
since it would have made notations more intricate; up to some 
integrability assumptions, this extension is straightforward. 
\end{Remark} 
\begin{Definition}
We call $U=(u,v) \in C([0,\infty),\bf{L}^2)$ a global finite energy 
solution to \eqref{divgen1}-\eqref{MBgen2} if   \eqref{nunu}  holds 
true and $U$ is a solution to \eqref{divgen1}-\eqref{MBgen2} 
in the distributional sense.
\end{Definition}
\begin{Remark}
Equation  \eqref{divgen1} has to be seen as a (linear) 
constraint, which  propagates from $t=0$ for solutions to 
\eqref{MBgen}-\eqref{MBgen2}:
\begin{eqnarray}
 \label{propagat} 
\dt (Id - P) (u -  (\kappa^{-1} \cdot l) \overline{v} ) =0 .
\end{eqnarray} 
Indeed, by definition of the projector $P$, we have $(Id - P)B=0$, 
so that we get \eqref{propagat} when applying $(Id - P)$ 
to \eqref{MBgen}, using  \eqref{MBgen2} (which extends 
to all $x  \in  \R^3$ since $F(x,0,u) \equiv 0$)  and commuting  
the derivative $\dt$ with $(Id - P)$ and $\kappa^{-1} \cdot l$. 
\end{Remark}
We  therefore have to consider initial data $U_{init}$ satisfying 
\eqref{divgen1}, and for such constrained initial data, 
the solutions to  \eqref{MBgen}-\eqref{MBgen2} also satisfy  
\eqref{divgen1} as long as they exist.
We shall write  $U_{init} := (u_{init} ,v_{init} ) $ with 
$u_{init} := (u_{init,1} , u_{init,2})$. 
\begin{Definition}
Let $\Ldiv$ be the set of functions $U := (u,v) \in 
L^2(\R^3,\R^6) \times (L^2(\Omega,\R^d) \cap L^\infty(\Omega,\R^d))$ 
satisfying \eqref{divgen1}.
\end{Definition}

\subsection{Cauchy problems}
\label{Cauchy}

Our first result states the existence of global finite  
energy solutions to \eqref{divgen1}-\eqref{MBgen2}.
\begin{Theorem}
\label{th:weak} 
Assume \eqref{hyp:kappa} and \eqref{affine}-\eqref{maj2}. 
For any $U_{init}  $ in $\Ldiv $, there exists  $U := (u,v) \in C([0,\infty),{\bf L}^2 )$, global finite energy solution to 
\eqref{divgen1}-\eqref{MBgen2} with  $U_{init} $ as initial data. 
Moreover, for all $T>0$, there is 
$C=C(T,F,l,\|v_{init} \|_{L^\infty})$ such that
\begin{enumerate}[(i)]
\item \label{estimweak1} for almost all $x \in \R^3$, for all $t\ge0$, 
$|v(t,x)| \leq |v_{init} (x)| e^{Kt} $ (with $K$ from \eqref{maj2});
\item \label{estimweak2} for all $t \in [0,T]$, 
$\| (u,v) (t)\|_{ {\bf L}^2} \leq C   \| U_{init} \|_{ {\bf L}^2}$;
\item \label{estimweak3} 
$v \in W^{1,\infty}_{loc}((0,\infty),L^2(\Omega,\R^d))$, 
and for almost every $t \in [0,T]$,
$\|\dt v (t)\|_{L^2 (\Omega)} \leq C \|U_{init}\|_{ {\bf L}^2}$.
\end{enumerate}
Finally, if $\mathcal{U}_{init}$ is a bounded set of $\Ldiv$ which 
is compact in $ {\bf L}^2$, then for all $T>0$, the set $\mathcal{U}$ 
of the above solutions with Cauchy data in $\mathcal{U}_{init}$ is 
compact in $C([0,T],{\bf L}^2 )$.
\end{Theorem}
%

To establish this first result, we follow the strategy of Jochmann 
in \cite{jochmann}, which is itself an improvement of the method 
by Joly, M\'etivier and Rauch in \cite{jmr}. 
This is the classical regularization method, in which (global-in-time) 
approximate solutions $U^n=(u^n,v^n)$ are built first 
(Section \ref{sec:approx}); the delicate step consists of course 
in passing to the limit $n\rightarrow\infty$ in the regularization 
(Section \ref{sec:compcomp}). Pointwise bounds are available 
for $v^n$, which imply $L^p$ bounds for 
$(Id - P)u^n = (Id - P)(\kappa^{-1}\cdot l)\overline{v^n}$, for 
finite $p$. The main argument relies on compensated compactness, 
applied to $Pu^n$ (Lemma~\ref{compactness}).

As a byproduct of the proof of Theorem \ref{th:weak}, we also have 
the following version of stability, where we assume strong convergence 
only for the $v$ part of the initial data. It will be useful below 
(cf. proof of Theorem~\ref{th:weakquasi}) when considering the 
weak quasi-stationary limit.
\begin{Proposition} 
 \label{weakprinciple}
Let  $(U^n)_{n\in\N}$ be a sequence in  
$ L^\infty_{loc} ((0,\infty ),\Ldiv ) $, bounded in  
$ L^\infty_{loc} ((0,\infty ), {\bf L}^2 )$, 
with  $(v^n)_{n\in\N}$ bounded in  
$W^{1,\infty}_{loc} ((0,\infty ),L^2 (\Omega)) \cap  
L^\infty_{loc} ((0,\infty ),L^\infty (\Omega)) $  
satisfying \eqref{MBgen2}, $v^n |_{t=0} \rightarrow v_{init} $ 
in $L^2 (\Omega)$ and $B u^n = \partial_t D^n$ 
with $(D^n )_n$ bounded in $ L^\infty_{loc} ((0,\infty ),L^2(\R^3)) $. Then, up to a subsequence, 
$v^n $ converges to $v$ in $L^\infty_{loc} ((0,\infty ),L^p(\Omega))$ 
for any $p \geqslant 2$ and in  
$L^\infty_{loc} ((0,\infty ),L^\infty (\Omega)) \text{ weak } *$,  
$u^n$ converges to $u$ in  
$L^\infty_{loc} ((0,\infty ),L^2 (\R^3)) \text{ weak } *$, 
and $U:= (u,v)$ satisfies  \eqref{divgen1}, \eqref{MBgen2}, 
as well as $v |_{t=0} = v_{init} $.
\end{Proposition} 

Let us now turn our atttention to smoother solutions. 
We need to assume more smoothness on the coefficients  $\eps$ and $\mu$. 
Of course, when considering in some physical situation a domain $\Omega $ with boundaries, 
the coefficients $\eps$ and $\mu$ experiment discontinuity jumps. 
Since we do not know how to tackle this physical case, we assume from now on  that 
\begin{equation}
\label{hyp:OmegaBded}
\Omega \text{ is bounded}, \quad \text{and with } 
\mathcal{K} = \overline{\Omega}, \quad
\kappa_i -1 \in \mathcal{C}^{\infty}_\mathcal{K} (\R^3 ), \, i=1,2. 
\end{equation}
In order to get a uniqueness result, we only need to ensure 
that the ``divergence free'' part $Pu$ of the fields has the 
$H^1$ regularity. To this end, once a finite energy solution is 
given, we make use of the linear system solved by $Pu$, with 
coefficients depending on the rest of the solution. 
As in \cite{jmr}, we proceed in two steps: we begin 
with the propagation of $H^\mu$ regularity, for $\mu\in(0,1)$, 
using Strichartz estimates (Proposition~\ref{StrichartzProppalim}). 
Applying this result with $\mu=1/2$ provides enough integrability 
for the coefficients of the above mentioned linear equation to 
ensure propagation of $H^1$ regularity. This implies that $u$ 
is ``almost'' $L^\infty$, a natural condition to prove 
uniqueness of the solution. Technically, a $L^\infty$ approximation 
of $u$ is built thanks to a limit Strichartz estimate for low 
frequencies (Proposition~\ref{StrichartzProp}). We also need a 
decoupling assumption, which was introduced in \cite{dumas}, and is 
satisfied by the Maxwell-Landau-Lifschitz system as well as by the 
Maxwell-Bloch system. 
 
\begin{Theorem}
\label{unique} 
In addition to the assumptions of Theorem \ref{th:weak}, 
assume \eqref{hyp:OmegaBded}. 
Let $\mu \in ]0,1]$, and $U_{init} \in \Ldiv$ with 
$\curl u_{init,i} \in H^{\mu -1 } (\R^3)$, for $i=1,2$. 
Then, the following holds true:
\begin{enumerate}[(a)]
\item \label{Unique1} Any solution $U$ to 
\eqref{divgen1}-\eqref{MBgen2} with  $U_{init}$ as initial data 
given by Theorem \ref{th:weak} satisfies 
$\curl u_{i} \in C([0,\infty ],H^{\mu -1}(\R^3))$, for $i=1,2$.
\item \label{Unique2} If $\mu = 1$, assuming moreover that
\begin{equation}
\label{decouplage}
\text{there exists $j \in \{1,2\}$ such that $l_{3-j} F  = 0$ 
and such that $F$ depends only on $(x,v,u_j)$},
\end{equation}
there exists only one solution to \eqref{divgen1}-\eqref{MBgen2} 
with  $ U_{init} $ as initial data as in Theorem \ref{th:weak}.
\end{enumerate}
\end{Theorem}

Theorem \ref{unique} asserts that the uniqueness property holds 
for initial data  $U_{init} $ in $L_{div}$ with $\curl  u_{init,i}  
\in L^2 (\R^3)$, $i=1,2$, which are dense in  $\Ldiv$ for the 
topology of ${\bf L^2  }$. 
The following theorem says that  the uniqueness property even 
holds generically for the following topologies.  
Let $\tau_s$ and  $\tau_w$ denote respectively the strong and 
weak topologies of $L^2(\R^3,\R^6 )$ and let $\tilde{\tau}_s$ 
denote the strong topology  of $L^2 (\Omega,\R^d)$. 
We consider the product  topology ${\bf \tau}_{ss} $ 
(resp. ${\bf \tau}_{ws} $) on  ${\bf L}^2$ 
obtained from $\tau_s$ (resp. $\tau_w$) and $\tilde{\tau}_s$.
\begin{Theorem}
\label{generic}
Under the assumptions of Theorem \ref{unique} \eqref{Unique2}, 
for any $C_{init} > 0$, there exists a $G_{\delta}$ dense set 
$\widetilde{L_{div}}$  in the set $\{ U_{init} \in \Ldiv 
\mid \| v_{init} \|_{L^\infty(\Omega)}  \leq C_{init}  \}$ 
for the  topology ${\bf \tau}_{ss} $ and ${\bf \tau}_{ws} $, 
such that for any $U_{init} \in \widetilde{L_{div}}$, 
there exists only one solution  to \eqref{divgen1}-\eqref{MBgen2} 
with  $ U_{init} $   as initial data, with the same properties 
as in Theorem~\ref{th:weak}.
\end{Theorem}
Let us stress that we cannot expect that the problem 
\eqref{divgen1}-\eqref{MBgen2} admits  smoother solutions 
than the ones given by Theorem \ref{unique} since, by definition,  
$\overline{v}$ is discontinuous across the boundary $\partial\Omega$.
However it follows from the general theory of discontinuous solutions 
of hyperbolic semilinear systems \cite{metivier,RR}
that the problem \eqref{divgen1}-\eqref{MBgen2} admits piecewise 
regular solutions discontinuous accross $\partial \Omega$ (let us also refer to the appendix of \cite{sueur}). 
Yet the general theory only guaranties  local-in-time solutions. 
We do not know if in the particular case of the problem 
\eqref{divgen1}-\eqref{MBgen2}  global-in-time solutions can 
be obtained.


\subsection{Quasi-stationary limits}
\label{Quasi-stationary limit}

As described in the Introduction, the quasi-stationary regime 
consists in the limit $\eta \rightarrow 0^+$ for \eqref{divgen1}, 
\eqref{MBgen}, \eqref{MBgen2}, where $\dt$ is replaced with 
$\eta\dt$, and $F$ is replaced with $\eta F$. Equations 
\eqref{divgen1} and \eqref{MBgen2} are invariant under this 
rescaling, whereas \eqref{MBgen} becomes 
\begin{eqnarray}
 \label{MBgeneta} 
 (\eta \dt + B ) u &=& \eta (\kappa^{-1} \cdot l)  F(x,\overline{v},u),
\quad \text{ for } x  \in  \R^3 .
\end{eqnarray}
For this semi-classical version of  \eqref{MBgen2}, it is still true that the constraint \eqref{divgen1} is propagated from the initial 
data. Formally, in the limit $\eta \rightarrow 0^+$, $v$ still 
satisfies \eqref{MBgen2}, whereas $u$ satifies \eqref{divgen1} and 
$Bu=0$. But for $U=(u,v) \in C([0,\infty),{\bf L}^2)$, these 
last two conditions are equivalent to the fact that for all 
$t\ge0$, $u(t)$ is directly determined by $v(t)$, and more precisely:
\begin{equation} \label{eq:limquasiu}
u = (Id - P) u = (Id - P) (\kappa^{-1} \cdot l) \overline{v}.
\end{equation}
Then, \eqref{MBgen2} becomes 
\begin{equation} \label{eq:limquasiv}
\dt v = F(x,v,(Id - P) (\kappa^{-1} \cdot l)v).
\end{equation}
Using the stability result given by Proposition~\ref{weakprinciple}, 
we have a first result of convergence towards the quasi-stationary 
limit, weakly for $u$ and locally in time for $v$:
\begin{Theorem}
\label{th:weakquasi} 
Assume \eqref{hyp:kappa}-\eqref{maj2}. 
For any $U_{init}  $ in $\Ldiv $, for any $\eta \in (0,1)$, let 
$U^\eta  := (u^\eta ,v^\eta)$ be a global finite energy solution 
to \eqref{divgen1}, \eqref{MBgen2} and \eqref{MBgeneta} 
with  $U_{init} $ as initial data. 
Then, up to a subsequence, $v^\eta $ converges in 
$L^\infty_{loc}((0,\infty),L^p(\Omega))$ for all $p \geqslant 2$ 
and in $L^\infty_{loc}((0,\infty),L^\infty(\Omega))\text{ weak }*$
towards a solution $v$ to \eqref{eq:limquasiv}, with $v_{init}$ 
as initial data; $Pu^\eta$ converges to $0$ in  
$L^\infty_{loc}((0,\infty),L^2(\Omega))\text{ weak }*$, and  
$(Id - P) u^\eta = (Id - P) (\kappa^{-1} \cdot l) \overline{v^\eta}$ 
converges in $L^\infty_{loc}((0,\infty ), L^p(\R^3))$ 
for all $p \geqslant 2$ towards $u$, given by \eqref{eq:limquasiu}.
\end{Theorem}
Convergence of the whole sequence $U^\eta $ is ensured as soon as 
the Cauchy problem associated with the limiting equation 
\eqref{eq:limquasiv} has a unique solution. This is given 
by the following proposition, which extends \cite[Theorem $3.1$]
{starynkevitchPhD} by Starynkevitch. 
\begin{Proposition} \label{prop:uniqquasi}
Assume \eqref{hyp:OmegaBded}, 
and let $v_{init} \in L^\infty (\Omega)$. 
Then,  there is a unique $v \in C([0,\infty),L^2(\Omega)) \cap 
L^\infty_{loc}((0,\infty ),L^\infty (\Omega))$ solution  
to \eqref{eq:limquasiv}, with $v_{init} $ as initial data. 
\end{Proposition}
We also prove strong and global-in-time convergence for $u$, 
assuming \eqref{hyp:OmegaBded} again, as well as 
integrability in time for $\| \dt v \|_{L^2(\Omega)}^2$ 
and non-trapping for some wave operator: 
\begin{Theorem}
\label{th:quasi} 
Under the assumptions of Theorem \ref{th:weakquasi}, 
assume moreover \eqref{hyp:OmegaBded}, the non-trapping 
hypothesis \eqref{hyp:nontrapping} and  that $\dt v^\eta$ 
is bounded (w.r.t. $\eta$) in $L^2((0,\infty)\times\Omega)$. 
Then, $P u^\eta$ goes to zero in $L^2((0,\infty ),L^2_{loc}(\R^3))$.
\end{Theorem}
\begin{Remark}
In the case of the Maxwell-Landau-Lifschitz system 
\eqref{maxwellLL}-\eqref{LL}, $\dt M$ actually belongs to 
$L^2((0,\infty)\times\Omega)$. Define the energy $\mathcal{E}(t)$ as 
$$
\mathcal{E}(t) = \frac{1}{2} \int_{\R^3} (\eps |E|^2 + \mu |H|^2)\, dx 
+ \int_\Omega \mu 
\left( \Phi(M) + \frac{1}{2} |H_{\rm ext}-M|^2 \right) dx . 
$$
Differentiating formally this expression with respect to time, 
we see that the integral of 
$H \cdot \curl E - E \cdot \curl H = \Div(E \wedge H)$ 
vanishes, as well as $M \cdot \dt M$ (since $|M|$ is constant). 
Using the orthogonality relations of the nonlinearity, we get 
$$
\dt M \cdot H_T = \alpha |M \wedge H_T|^2  
\quad \text{and} \quad
|\dt M|^2 = (\alpha^2+\gamma^2) |M \wedge H_T|^2,
$$
so that estimate $(ii)$ in Theorem~\ref{th:weak} is improved to 
$$
\mathcal{E}(t) + \frac{\alpha}{\alpha^2+\gamma^2} \int_0^t   
\|  \sqrt{\mu} \dt M(t') \|_{L^2(\Omega)}^2 dt' 
- \frac{1}{2} \| \sqrt{\mu} H_{\rm ext}(t) \|_{L^2(\Omega)}^2 
+ \int_0^t \left( \int_\Omega \mu M\cdot\dt H_{\rm ext} dx \right) dt'  = cst,
$$
and the same is true with the quasi-stationary scaling. 
Assuming for example that $H_{\rm ext} \in L^\infty_tL^2_x$ 
and $\dt H_{\rm ext} \in L^1_{t,x}$, we deduce that $\mathcal{E}$ 
is bounded, and $\dt M$ belongs to $L^2((0,\infty)\times\Omega)$. 

In the case of the Maxwell-Bloch system, we do not know if such an 
estimate is available for $\dt\rho$.
\end{Remark}
%

  
\section{Existence of global finite energy solutions: 
proof of Theorem~\ref{th:weak}}
\label{proofweak} 

\subsection{Technical interlude $1$}
\label{prelim} 

\subsubsection{Intersections and sums of Banach spaces}

 We recall  some useful properties of  the intersection and  the sum of Banach spaces.
Consider two Banach spaces $X_1$ and $X_2$ that are subsets of  a Hausdorff topological vector
space  $X$. Then
\begin{eqnarray*}
& X_1 \cap X_2 :=  \{ f  \in X \mid \ f \in X_1 , \ f \in X_2 \} \\
& \ (\text{respectively } X_1 + X_2 :=  
\{ f  \in X \mid  \ E(f)  \neq \emptyset \} , 
 \text{ where } E(f):= 
 \{ (f_1 ,f_2 ) \in X_1  \times  X_2  \mid \  f_1 + f_2 = f \})
\end{eqnarray*}
is a Banach space endowed with the norm
\begin{eqnarray*}
& \| f  \|_{X_1  \cap  X_2 } := \| f \|_{ X_1 } + \| f \|_{ X_2} \\
& \   (\text{respectively } \| f  \|_{X_1 + X_2 } :=  \inf \{ \| f_1 \|_{ X_1 } + \| f_2 \|_{ X_2} \mid (f_1 ,f_2 ) \in E(f) \}).
\end{eqnarray*}
If furthermore $X_1  \cap  X_2$ is a dense subset of both $X_1$ and $  X_2$, then
$ (X_1  \cap  X_2)' = X'_1  +  X'_2$ and  $(X_1  +  X_2 )' =X'_1  \cap  X'_2 $  (cf. Bergh and L{\"o}fstr{\"o}m \cite{BergLof}, Lemma $2.3.1$ and Theorem $2.7.1$).

\subsubsection{Mollifiers}
   
 We shall use the following symmetric operators 
$R^n : L^2 (\R^3) \rightarrow  L^2 (\R^3) $, defined by
\begin{eqnarray} 
 \label{formulus}
(R^n f)(x) :=  \int_{\R^3 } f(y) w^n (x-y) dy  \quad 
\text{ for } x \in \R^3 , 
\end{eqnarray}
where $w^n \in C^\infty_0  (\R^3)$ is a mollifier with $\text{ supp } w^n \subset B(0,1/(1+n))$ 
and $\int_{\R^3 } w^n = 1$. These operators have the following well-known properties: 
there exists $C>0$ such that  for all $f \in  L^2 (\R^3) $, $r > 1$ and  $n \in \N$, 
\begin{eqnarray} 
 \label{opreg1}
\| f - R^n f \|_{L^2 (\R^3)}  \rightarrow 0 , & \ &  \|  R^n f \|_{L^2 (\R^3)}  \leqslant C  \|   f \|_{L^2 (\R^3)} , \\ 
\label{opreg2} 
 \|  R^n f ||_{L^2 (B_r )}  \leqslant C  \|   f \|_{L^2 (B_{ r+1 } )} ,  &\text{ and } &
  \|  R^n f \|_{L^2 (\R^3 \setminus B_r )}  \leqslant C  \|   f \|_{L^2 ( \R^3 \setminus B_{ r-1 } )}.
\end{eqnarray}
 Moreover for all  $n \in \N$, there exists $C_n >0$ such that  for all $f \in  L^2 (\R^3) $,
\begin{eqnarray} 
\label{opreg3}
 \|  R^n f \|_{L^\infty (\R^3  )}  \leqslant C_n  \|   f \|_{L^2 ( \R^3  )}.
\end{eqnarray}

\subsection{Approximate solutions}
\label{sec:approx}

The following lemma claims the existence of global solutions to some regularized problem.
\begin{Lemma} \label{lem:bornesregMBgen}
 For all $n \in \N$,  there exists  $U^n := (u^n,v^n ) \in C([0,\infty),  {\bf L}^2 )$,  with  
\begin{eqnarray}
\label{formula88}
 v^n  \in    C ([0,\infty), L^\infty ( \Omega,\R^d )) \cap C^1 ([0,\infty),  L^2 ( \Omega,\R^d )),
\end{eqnarray}
solution to the regularized problem:
\begin{eqnarray}
 \label{regMBgen1} 
 ( \dt + B ) u^n &=& 
(\kappa^{-1} \cdot l)  F^n \quad \text{ for } x  \in  \R^3 ,
\\ \label{regMBgen2} 
 \dt v^n &=& F^n \quad  \text{ for } x  \in \Omega ,
\end{eqnarray}
where 
\begin{eqnarray}
\label{formula}
F^n (t,x) :=  F(x,\overline{v^n}  (t,x), R^n u^n  (t,x)),
\end{eqnarray}
with  $U_{init} $ as initial data. 
Moreover, for all $n \in \N$, 
\begin{enumerate}[(a)]
\item  \label{borne1}   For almost all $x \in \R^3$, for all $t\ge0$, 
$|v^n (t,x)| \leq |v_{init}(x)| e^{Kt}$ (with $K$ from \eqref{maj2}).
\item  \label{borne2}   For all $T>0$, there is 
$C=C(T,F,l,\kappa,\|v_{init}\|_{L^\infty})$ such that, 
for all $t \in [0,T]$,  
$\|(u^n,v^n )(t)\|_{{\bf L}^2} + \|  \dt v^n (t)\|_{L^2 (\Omega)} 
\leq C \| U_{init} \|_{{\bf L}^2}$.
\end{enumerate}
\end{Lemma}
\begin{proof}
The local-in-time solution is constructed \emph{via} a usual fixed point argument 
for the mapping $ \mathcal{A}^n : C ([0,T],{\bf L}^2) \rightarrow C ([0,T],{\bf L}^2)$,
$$
\mathcal{A}^n (u,v) (t,\cdot) = \Big( \exp( -tB) u_{init} + \int_0^t  \exp((t-t')B)  (\kappa^{-1} \cdot l)  
\mathcal{F}^n (t',\cdot) dt' , \, v_{init} + \int_0^t \mathcal{F}^n (t',\cdot) dt' \Big) ,
$$
where 
\begin{eqnarray*}
\mathcal{F}^n (t,\cdot) :=  F(\cdot ,\overline{v}  (t,\cdot), R^n u  (t,\cdot )).
\end{eqnarray*}
For $T>0$ small enough, $\mathcal{A}^n$ is shown to be a contraction mapping 
 thanks to properties 
\eqref{affine}-\eqref{maj2} of $F$, \eqref{opreg3}, and because $B$ is a skew self-adjoint operator 
in the Hilbert space $L^2 (\R^3 ,  \R^6)$ endowed with the scalar product 
$\langle \cdot  , \cdot  \rangle_{\kappa_1 , \kappa_2 } $. 

Global existence is given by the \emph{a priori} bounds (a) and (b). 
The first one follows directly from \eqref{maj2} and Gronwall's lemma. In the same way, taking the 
$L^2$ norm of $(u^n,v^n)(t) = \mathcal{A}^n (u^n,v^n)(t)$, one gets 
$$
\| (u^n,v^n )(t)\|_{ {\bf L}^2} \le \| U_{init} \|_{ {\bf L}^2} 
+ \int_0^t ( 1 + \| \kappa^{-1} \cdot l \|_{L^\infty} ) 
\| F^n(t',\cdot) \|_{L^2} dt' .
$$
One may add to this inequality the one obtained from \eqref{regMBgen2}, 
$$
\| \dt v^n(t,\cdot) \|_{L^2} \le \| F^n(t,\cdot) \|_{L^2} . 
$$
From \eqref{affine}, \eqref{maj1}, we have
$$
\| F^n(t,\cdot) \|_{L^2} \le C_F(\|v_{init}\|_{L^\infty} e^{Kt}) 
\|(u^n,v^n )(t)\|_{ {\bf L}^2} , 
$$
so that Gronwall's lemma concludes.
\end{proof}

\subsection{Passing to the limit $n  \rightarrow \infty$}
\label{sec:compcomp}

Let us stress that Estimate \eqref{opreg3} is not uniform with respect to $n$. However, we have: 

\begin{Proposition} 
\label{prop:CVforte}
For all $T>0$, there is a subsequence of $(U^n)_{n\in\N}$ given by 
Lemma~\ref{lem:bornesregMBgen} that strongly converges in 
$C([0,T],{\bf L}^2)$ to  $U := (u,v) \in C([0,\infty),{\bf L}^2)$,  
global finite energy solution to \eqref{divgen1}-\eqref{MBgen2} with  
$U_{init} $ as initial data, and satisfying  the estimates  
\eqref{estimweak1}, \eqref{estimweak2}, \eqref{estimweak3} 
of Theorem~\ref{th:weak}. 
\end{Proposition} 

\begin{proof}
First we infer from the bounds \eqref{borne1}-\eqref{borne2} in Lemma~\ref{lem:bornesregMBgen} that there exists a subsequence, 
still denoted $(u^n,v^n)$, such that $ u^n$ (respectively $F^n$) 
tends to $u$ (resp. to $F_\text{lim}$) in $L^\infty((0,T),{\bf L}^2)   
\text{ weak } *$ (resp. $L^\infty((0,T),L^2(\R^3)) \text{ weak }*$) 
and $v^n$ tends to $v$ in  $W^{1,\infty}((0,T),L^2 (\Omega)) 
\text{ weak }*$ and in  $L^\infty ((0,T),L^\infty(\Omega)) 
\text{ weak }*$. This is enough to ensure that $(u,v)$ satisfies  
\eqref{divgen1}. 
Moreover, Fatou's lemma yields that $u$ and $v$ satisfy 
\eqref{estimweak1}-\eqref{estimweak2} of Theorem \ref{th:weak} 
for almost every $t$ in $(0,T)$ .

Since the function $F$ is not linear, these weak limits do not 
suffice to pass to the limit in Equation  \eqref{regMBgen2}. 
The strategy is to carefully study the nonlinear term $F^n$ to 
prove that the solutions $U^n$ of the regularized problems  
\eqref{regMBgen1}-\eqref{regMBgen2} actually converge (strongly) 
in $L^2 $. The key step consists in proving the strong convergence 
of $v^n$. 

It shall be useful several times to keep in mind that, thanks to the growth conditions  \eqref{maj1} on $F$ and to the pointwise bound Lemma~\ref{lem:bornesregMBgen}, \eqref{borne1} of the $v^n$, there holds, for all $n,m \in \N$, for all $(t,x) \in \lbrack 0,T \rbrack  \times \R^3$, 
\begin{eqnarray} \label{avant1}
| F_i^n(t,x) | & \leqslant & C_F (e^{KT} \|v_{init}\|_{L^\infty}), \\  
\label{avant2} 
| F_i^n (t,x) - F_i^m (t,x) | & \leqslant &  
C_F (e^{KT} \|v_{init}\|_{L^\infty}) 
|\overline{v^n}(t,x) - \overline{v^m}(t,x)|,
\end{eqnarray}
where $F_i^n=F_i (x,\overline{v^n})$, $i=0,1$. 

\emph{Strong convergence of $v^n$.} We perform energy estimates on 
\eqref{regMBgen1}-\eqref{regMBgen2}. Since $u$ may 
be unbounded, we introduce a weight function, which precisely depends on  $u$. 
More exactly, we choose a positive function $\rho_0 (x)$ in  $L^\infty (\R^3) \cap L^2 (\R^3)$ and define
\begin{eqnarray} \label{defpoids}
\rho(t,x) := \rho_0 (x) e^{-L \int_0^t  |u (s,x)| ds  } ,
\end{eqnarray}
with $L \ge C_F (e^{KT} \|v_{init}\|_{L^\infty})$. 
First, using \eqref{regMBgen2} we get
\begin{eqnarray*}
 \frac{1}{2} \frac{\rm d}{\rm dt} \left( \|  \rho ( v^n - v^m )  \|^2_{L^2 (\Omega ) } \right) (t) =
   \int_\Omega  \rho^2 ( v^n - v^m ) \cdot (F^n - F^m ) dx 
   - L  \int_\Omega  \rho^2   |u| | v^n - v^m |^2 dx   .
\end{eqnarray*}
Next, decompose $F^n - F^m $ according to \eqref{affine} to get 
\begin{equation}
 \label{poids}
\begin{split} 
 \frac{1}{2} \frac{\rm d}{\rm dt} \left(  \|  \rho ( v^n - v^m )  \|^2_{L^2 (\Omega ) } \right) (t) = 
 & \int_\Omega  \rho^2 ( v^n - v^m ) \cdot (F_0^n  - F_0^m ) dx \\
 & + \int_\Omega  \rho^2 ( v^n - v^m ) \cdot (F_1^n R^n u^n - F_1^m R^m u^m ) dx 
   - L \int_\Omega  \rho^2   |u| | v^n - v^m |^2 dx  .
\end{split}
\end{equation}
The first term in the r.h.s. of  \eqref{poids} can be estimated by $C  \|  \rho ( v^n - v^m )  \|^2_{L^2 (\Omega ) } (t)$ thanks to \eqref{avant2}.
Now, decompose $F_1^n R^n u^n - F_1^m R^m u^m$ into
\begin{eqnarray*}
F_1^n R^n u^n - F_1^m R^m u^m = 
 F_1^n \big( R^n u^n - u  \big) 
-  F_1^m  \big( R^m u^m - u \big)  
+  \big(F_1^n   -  F_1^m  \big) u  . 
\end{eqnarray*}
The terms produced by the third parenthesis are estimated thanks to \eqref{avant2},
and absorbed by the last term in \eqref{poids}, so that
\begin{eqnarray*}
\frac{1}{2} \frac{\rm d}{\rm dt} 
\left( \|  \rho ( v^n - v^m )  \|^2_{L^2 (\Omega ) } \right) (t) 
&\leqslant& C  \|  \rho ( v^n - v^m )  \|^2_{L^2 (\Omega ) } (t)
+ \int_\Omega  \rho^2 ( v^n - v^m ) \cdot F_1^n ( R^n u^n -  u)  dx \\  
& & + \int_\Omega \rho^2 ( v^m - v^n ) \cdot F_1^m (R^m u^m - u) dx .
 \end{eqnarray*}
Then, decompose $R^n u^n$ and $R^m u^m$ according to the orthogonal 
projector $P$ to get 
\begin{eqnarray}
\label{estimenergv}
\frac{1}{2} \frac{\rm d}{\rm dt} 
\left( \|  \rho ( v^n - v^m )  \|^2_{L^2 (\Omega ) } \right) (t) 
\le C  \|  \rho ( v^n - v^m )  \|^2_{L^2 (\Omega ) }(t) 
+  \sum_{j=1}^3 h_{j,m,n} (t) +  h_{j,n,m} (t) ,
\end{eqnarray}
where
\begin{eqnarray*}
h_{1,m,n} (t) 
&:=& \int_\Omega \rho^2 (v^n - v^m) \cdot F_1^n R^n P (u^n-u) dx , \\ 
h_{2,m,n} (t) 
&:=& \int_\Omega \rho^2 (v^n - v^m) \cdot 
F_1^n R^n (Id - P) (u^n-u) dx , \\   
h_{3,m,n} (t) 
&:=&  \int_\Omega  \rho^2 ( v^n - v^m ) \cdot F_1^n (R^n u -  u) dx .
\end{eqnarray*}

The following lemma  deals with the term $ h_{1,m,n} (t)$.
\begin{Lemma} \label{h1n1}
There holds
\begin{equation}
\label{estimh1}
\forall \delta>0, \, \exists N_\delta\in\N, \, \forall n \geqslant N_\delta, \, 
\forall m \in \N, \qquad \Big| \int_0^T h_{1,m,n} (t) dt \Big| \le 2 \delta . 
\end{equation}
 \end{Lemma}
\begin{proof}
First notice that 
$$h_{1,m,n} (t) =   \int_{\R^3}  R^n \Big( \rho^2 ( \overline{v^n} - \overline{v^m} )  
\cdot F_1^n \Big) P ( u^n -  u) dx .$$

We first handle the case where $x$ is outside of a large ball.
Using the Cauchy-Schwarz inequality, the second property of  $R^n$ in \eqref{opreg2}, the uniform bound in $L^\infty ( \lbrack 0,T \rbrack , L^2(\R^3))$ for $v^n$ given   in Lemma~\ref{lem:bornesregMBgen}, \eqref{borne2} and the bound \eqref{avant1} for $F_1^n$, we get that 
\begin{equation}
\int_0^T \int_{\R^3 \setminus B_r} \Big| R^n \Big( \rho^2 ( \overline{v^n} 
- \overline{v^m} )  F_1^n \Big) \cdot P (  u^n -  u) \Big| dx dt 
\le C  \int_0^T \|\rho(t)^2\|_{L^2(\R^3 \setminus 
B_{r-1})} dt .
\end{equation}
By definition of $\rho$ there exists  $r>0$ so that this integral 
is less than $\delta$. 

It remains to tackle the case where $x \in B_r$.
We use the following compactness lemma:
\begin{Lemma}[Jochmann  \cite{jochmann}, Lemma 3.4]
 \label{compactness} 
Let $(G^n )_{n \in \N}$ and $(K^n )_{n \in \N}$ be bounded sequences in 
$ L^\infty (\lbrack 0,T),L^2 (\R^3  ,  \R^6))$, with $K^n$ converging to $0$  in $ L^\infty 
(\lbrack 0,T),L^2 (\R^3  ,  \R^6))  \text{ weak } *$.
Suppose that $(G^n )_{n \in \N}$ is equicontinuous from $[0,T]$ to $L^2 (\R^3  ,  \R^6)$ 
and that $B K^n = \dt C^n$ with  $(C^n )_{n \in \N}$  
bounded in $ L^\infty (\lbrack 0,T),L^2 (\R^3  ,  \R^6))$. Then for all $r >0$,
\begin{eqnarray} 
\sup_{p \in \N} \Big| \int_0^T  \int_{B_r} G^p (t) \cdot P K^n (t) dx dt \Big|  
\tendlorsque{n}{\infty} 0 .
\end{eqnarray}
 \end{Lemma}
Let us denote  $G^{k,l} = 
R^k \Big( \rho^2 (\overline{v^k}-\overline{v^l}) \cdot F_1^k \Big)$ 
and  $K^n = u^n - u$. 
Thanks to \eqref{avant1}, \eqref{opreg1}  and to 
Lemma~\ref{lem:bornesregMBgen}, \eqref{borne2},  
we get that  $(G^{k,l})_{k,l \in \N}$ and $(K^n )_{n \in \N}$
are bounded in $L^\infty ([0,T),L^2 (\R^3))$.
Moreover  $K^n$ tends to zero in $ L^\infty ([0,T),L^2 (\R^3))  \text{ weak } *$, by definition of $u$. 
Let us denote $\mathcal{F}^n = \int_0^t (\kappa^{-1} \cdot l) F^n dt'$ and  $\mathcal{F} = \int_0^t (\kappa^{-1} \cdot l) F_\text{lim}  dt'$.
From  \eqref{regMBgen1} we infer that
$$B K^n = \dt C^n , \quad {\rm with} 
\quad C^n := \mathcal{F}^n - u^n -  ( \mathcal{F} - u ) .$$
  The sequence 
$(C^n )_{n \in \N}$ is bounded in $L^\infty ([0,T),L^2 (\R^3))$. 
In the same way, equicontinuity is obtained from the bounds on $\dt v^n = F^n$. 
We therefore apply the lemma observing that, for all $m,n \in \N$, 
$$\Big| \int_0^T  \int_{B_r} G^{m,n} (t) \cdot P K^n (t) dx dt \Big| \le 
\sup_{k,l \in \N} \Big| \int_0^T  \int_{B_r} G^{k,l} (t) \cdot P K^n (t) dx dt \Big| . $$ 

Lemma~\ref{compactness} therefore ensures that there is 
$N_{r,\delta}\in\N$ such that, for $n \geqslant N_{r,\delta}$ and for all $m \in \N$, 
$$\Big|\int_0^T \int_{B_r} R^n \Big( \rho^2 ( \overline{v^n} 
- \overline{v^m} )  F_1^n \Big) \cdot P (  u^n -  u) dx dt \Big| \le \delta ,$$
and Lemma \ref{h1n1} is proved.
\end{proof}

We now deal  with the term $ h_{2,m,n} (t)$.
\begin{Lemma} \label{h1n2}
There holds
\begin{equation}
\label{estimh2}
\forall \delta>0, \, \exists N_\delta\in\N, \, \forall n \ge N_\delta, \, 
\forall m \in \N, \qquad \Big| \int_0^T h_{2,m,n} (t) dt \Big| 
\le \delta + C \| \rho ( \overline{v^n}-\overline{v^m} ) \|_{L^2_{t,x}} 
\| \rho ( \overline{v^n}-\overline{v} ) \|_{L^2_{t,x}}.
\end{equation}
 \end{Lemma}
\begin{proof}
The $(u^n ,v^n )$ satisfy \eqref{divgen1}  and so does their weak limit $(u ,v )$. Thus
\begin{eqnarray*}
h_{2,m,n} (t) 
&=& - \int_\Omega  \rho^2 ( v^n - v^m ) \cdot F_1^n R^n 
(Id - P) ( \kappa^{-1} \cdot l) ( \overline{v^n}-\overline{v} ) dx  \\
&=& - \int_{\R^3}  \rho^2 ( \overline{v^n} - \overline{v^m} ) \cdot F_1^n R^n 
(Id - P) ( \kappa^{-1} \cdot l) ( \overline{v^n}-\overline{v} ) dx .
\end{eqnarray*}
Then we decompose 
\begin{equation} \label{h2split}
\begin{split}
\int_0^T | h_{2,m,n} (t) | dt \le 
& \int_0^T \Big| \int_{\R^3} ( \overline{v^n} - \overline{v^m} ) \cdot F_1^n 
\Big[ \rho^2 R^n (Id - P) ( \kappa^{-1} \cdot l) ( \overline{v^n}-\overline{v} ) \\
& \qquad \qquad \qquad - \rho R^n (Id - P) \rho 
( \kappa^{-1} \cdot l) ( \overline{v^n}-\overline{v} ) \Big] dx \Big| dt \\
& + \int_0^T \Big| \int_{\R^3} \rho ( \overline{v^n} - \overline{v^m} ) \cdot F_1^n 
R^n (Id - P) \rho ( \kappa^{-1} \cdot l) ( \overline{v^n}-\overline{v} ) dx \Big| dt .
\end{split}
\end{equation}
The second integral in the r.h.s. of \eqref{h2split} is estimated 
thanks to H\"older's inequality:
\begin{equation*}
\int_0^T \Big| \int_{\R^3} \rho ( \overline{v^n} - \overline{v^m} ) \cdot F_1^n 
R^n (Id - P) \rho  ( \kappa^{-1} \cdot l) ( \overline{v^n}-\overline{v} ) dx \Big| dt \le 
C \| \rho ( \overline{v^n}-\overline{v^m} ) \|_{L^2_{t,x}} 
\| \rho ( \overline{v^n}-\overline{v} ) \|_{L^2_{t,x}},
\end{equation*}
where $C$ depends only on $T,F,l,\|v_{init}\|_{L^\infty}$.
To deal with the first integral in the right-hand side of \eqref{h2split}, we use the following commutation lemma.
\begin{Lemma}[Jochmann  \cite{jochmann}, Lemma 3.5]
 \label{commutator} 
 Let $\rho$ belong to $L^2 ((0,T),L^2  (\R^3)) \cap L^{\infty} ((0,T),L^\infty (\R^3))$, 
and let $(M^n )_{n \in \N}$ be a bounded sequence in $W^{1,\infty} ((0,T),L^2  (\R^3)) 
\cap L^{\infty} ((0,T),L^\infty  (\R^3))$ which converges to $0$ in 
$ L^\infty ((0,T),L^2 (\R^3))  \text{ weak } *$. Then
\begin{eqnarray} 
\int_0^T  \|\rho(t)^2 R^n (Id - P) M^n - \rho(t) R^n (Id - P) \rho(t) M^n  
\|_{L^1 (\R^3) + L^2 (\R^3) }  dt   \tendlorsque{n}{\infty} 0 , \label{commutator1} \\  
\mbox{and} \quad \int_0^T  \|\rho(t)^2  (Id - P) M^n - \rho(t)  (Id - P) \rho(t) M^n    
\|_{L^1 (\R^3) + L^2 (\R^3) }  dt  \tendlorsque{n}{\infty} 0 \label{commutator2} .
\end{eqnarray}  
\end{Lemma}
We apply Lemma~\ref{commutator}, \eqref{commutator1} with 
$M^n = ( \kappa^{-1} \cdot l) ( \overline{v^n}-\overline{v} )$: 
the first integral in the right-hand side of \eqref{h2split} is estimated by  
$$C(T,F,\|v_{init}\|_{L^\infty}) 
\int_0^T \| \overline{v^n}-\overline{v^m} \|_{L^\infty \cap L^2} 
\| \rho^2 R^n (Id-P) M^n  \\
- \rho R^n (Id - P) \rho M^n \|_{L^1+L^2} dt ,$$
and thus goes to zero as $n$ goes to infinity, uniformly w.r.t. $m$. 
Hence, we get  \eqref{estimh2}.
\end{proof}

For all $\delta>0$, we also bound $h_{3,m,n}$ by $\delta$ for all $n \ge N_\delta$ 
and all $m \in \N$ thanks to \eqref{opreg1}. Finally, summing up with \eqref{estimh1} 
and \eqref{estimh2}, we have from \eqref{estimenergv}:
\begin{equation}
\label{estimvp-vm}
\begin{split}
& \forall \delta>0, \, \exists N_\delta\in\N, \, \forall n,m \geqslant N_\delta,  \, 
\forall t \in [0,T], \\
& \|  \rho ( v^n - v^m )  \|^2_{L^2 (\Omega ) } (t) 
\le C \left( \delta + \|  \rho ( v^n - v^m )  \|^2_{L^2((0,T)\times\Omega)}  
+ \|  \rho ( v^n - v )  \|^2_{L^2((0,T)\times\Omega)} \right).
\end{split}
\end{equation}  
Use Gronwall's Lemma, then let $m$ go to $\infty$, and use Gronwall's Lemma again 
to deduce:
$$
\forall \delta>0, \, \exists N_\delta\in\N, \, \forall n \geqslant N_\delta,  \, 
\forall t \in [0,T], \qquad 
\|  \rho ( v^n - v )  \|^2_{L^2 (\Omega ) } (t) \le C \delta ,
$$
which implies that $v^n$ converges towards $v$ strongly 
in $L^2((0,T)\times\Omega,\rho(t,x)^2 dt dx)$. 
Up to a subsequence, convergence then holds almost everywhere, 
for the measure $\rho^2 dt dx$, or $dt dx$, since $\rho$ 
is positive almost everywhere in $(0,T)\times\Omega$. 
Thanks to the pointwise estimates \eqref{borne1} from 
Lemma~\ref{lem:bornesregMBgen}, dominated convergence 
thus ensures that $v^n$ converges towards $v$ strongly 
in $L^2((0,T)\times\Omega,dt dx)$. 
Then, equicontinuity of $\{v\} \cup \{v^n\}_{n\in\N}$ 
in $C([0,T],L^2(\Omega))$ implies (by Ascoli's Theorem) 
the strong convergence of $v^n$ in $C([0,T],L^2(\Omega))$. 
This, together with the uniform bounds on 
$\{v\} \cup \{v^n\}_{n\in\N}$ and with the weak 
convergence of $u^n$, is enough to pass to the limit 
in \eqref{regMBgen1}, \eqref{regMBgen2} to get 
\eqref{MBgen}, \eqref{MBgen2}. \\
\\ 
\emph{Strong convergence of $u^n$.} 
Since $u^n$ and $u$ satisfy \eqref{regMBgen1} and \eqref{MBgen} respectively, 
their difference is solution to a hyperbolic equation with source term 
in $L^1((0,T),L^2(\R^3))$,
$$( \dt + B ) (u^n-u) = (\kappa^{-1} \cdot l)  (F^n-F(x,\overline{v},u)).$$
The standard energy estimate then gives 
\begin{equation}
\label{enun-u}
\begin{split}
\| u^n-u \|_{L^2}(t) 
& \le C \int_0^t \| F^n-F(x,\overline{v},u) \|_{L^2}(t') dt' \\
& \le C \int_0^t \left( \| F_0^n - F_0(x,\overline{v}) \|_{L^2}(t')  
+ \| F_1^n  R^n u^n - F_1(x,\overline{v}) u \|_{L^2}(t') \right) dt' .
\end{split}
\end{equation}

Thanks to the growth conditions  \eqref{maj1} of $F$ and to the pointwise bound Lemma~\ref{lem:bornesregMBgen}, \eqref{borne1} of the $v^n$, there holds, for all $n,m \in \N$ 
and $(t,x) \in \lbrack 0,T \rbrack  \times \R^3$, 
\begin{eqnarray}
 \label{avant3} | F_i^n (t,x) - F_i  ( x,\overline{v} (t,x)) | &\leqslant&  C_F ( e^{KT} \|v_{init}\|_{L^\infty}   ) |\overline{v^n}  (t,x) - \overline{v}  (t,x)|.
\end{eqnarray}
In particular this yields that for any $t'  \in \lbrack 0,T \rbrack $, 
$\|F_0^n  - F_0(x,\overline{v})\|_{L^2} (t')$
 goes to zero as $n$ goes to infinity. 

Furthermore, 
$$|F_1^n  R^n u^n - F_1(x,\overline{v}) u| \le 
|F_1^n | |R^n (u^n-u)| + |F_1^n | |(R^n-Id)u| 
+ |F_1^n  - F_1(x,\overline{v})| |u|.$$
Thanks to the $L^\infty$ bounds on $v^n$ (cf.  Lemma~\ref{lem:bornesregMBgen}, \eqref{borne1}) and on $F_1^n$ (cf. \eqref{avant1}), and to the property 
\eqref{opreg1} of the operator $R^n$, the first term 
in the r.h.s. above is bounded by $C_F ( e^{KT} \|v_{init}\|_{L^\infty}   ) |u^n-u|$. 
In the same way, the second term goes to zero in $L^2$ as $n$ goes to infinity. 
Finally, up to a subsequence, the third term tends to zero almost everywhere, 
and is bounded by $C(F,T,\| v_{init} \|_{L^\infty}) |u|$. By dominated convergence, 
it thus goes to zero in $L^2$. 
Finally, we get from \eqref{enun-u}:
$$\| u^n-u \|_{L^2}(t) \le C(F,T,\| v_{init} \|_{L^\infty}) 
\int_0^t \| u^n-u \|_{L^2}(t') dt' + o(1),$$
and Gronwall's Lemma shows that $u^n$ converges to $u$ in $C([0,T],L^2)$.
\end{proof}

Thanks to a diagonal extraction process (using times $T\in\N^\star$), 
Proposition~\ref{prop:CVforte} produces $U \in C([0,\infty),L^2)$, 
solution to \eqref{divgen1}-\eqref{MBgen2} with  $U_{init} $ as initial data. 
Estimates \eqref{estimweak1} and \eqref{estimweak2} are then straightforward. 
To prove Theorem~\ref{th:weak}, there remains to show its last statement: the stability property. 
To this end, consider a sequence $(U_{init}^n)_{n\in\N}$, bounded in $\Ldiv$, 
and converging to $U_{init}$ in $L^2$. It generates a (sub)sequence of solutions 
$(U^n)_{n\in\N}$, with, from the bounds \eqref{estimweak1}, \eqref{estimweak2} 
in Theorem~~\ref{th:weak}, $u^n$ converging to  $u$ in  $ L^\infty_{loc} ((0,\infty ),L^2 (\R^3)) 
\text{ weak } *$ and $ v^n$ converging to  $v$ in  $W^{1,\infty}_{loc} ((0,\infty ),L^2 (\Omega))$ 
weakly * and in  $ L^\infty_{loc} ((0,\infty ),L^\infty (\Omega))   \text{ weak } *$. Then, define 
the weight $\rho$ from \eqref{defpoids} and estimate $v^n -v^m$ as in \eqref{poids}, 
with $R^n u^n $ and $R^mu^m$ replaced with $u^n $ and $u^m$, respectively. This leads 
to the analogue to \eqref{estimenergv}, with no $h_{3,m,n }$ and $h_{3,n ,m}$ terms, 
and no $R^n $ in $h_{1,m,n }$ and $h_{2,m,n }$. Apply Lemma~\ref{compactness} and 
Lemma~\ref{commutator}, \eqref{commutator2} (instead of \eqref{commutator1}), to 
get strong $C([0,T],L^2)$ convergence of $v^n$ towards $v$ (in \eqref{estimvp-vm}, 
the term $\|v_{init}^n -v_{init}^m\|_{L^2}$ goes to zero, and contributes to $\delta$). 
Strong convergence of the fields $u^n$ is then obtained as above , with an initial 
term $\|u_{init}^n-u_{init}\|_{L^2}$ going to zero added to the r.h.s. of 
\eqref{enun-u}.

The same process proves Proposition \ref{weakprinciple}.

 
\section{Propagation of smoothness and uniqueness: 
proof of Theorem~\ref{unique}}
\label{proofunique} 

It is worth noting that, under the smoothness assumption on $\eps$ 
and $\mu$ in \eqref{hyp:OmegaBded}, $u  \in L^2  (\R^3 ,\R^6 )$ 
with $Pu  \in H^\mu  (\R^3 ,\R^6)$ iff $u  \in L^2  (\R^3 ,\R^6 )$ 
with $\curl u_i  \in H^{\mu-1}  (\R^3 ,\R^6)$ for $i=1,2$.

We thus split the proof of Theorem \ref{unique} in several steps. 
In Section \ref{startPre}, we isolate a Cauchy  problem for the projection $Pu$ of $u$.  
This allows some dispersive estimates that we etablish in 
Section \ref{Disper}, while in Section \ref{prelim2}, 
Littlewood-Paley decompositions are introduced. 
We consider first the case where $\mu $ is in $(0,1)$, then we 
prove the part \eqref{Unique1} of  Theorem ~\ref{unique} in the 
case  $\mu = 1$, which concerns the propagation of smoothness, 
and finally the part  \eqref{Unique2}, which concerns uniqueness. 

\begin{Remark}
Let us mention that in the proof of the propagation of $H^1$ 
regularity given in  \cite{dumas}, the step ``$\mu\in(0,1)$'' 
is missing, and the resulting estimates (collected here in 
Lemma~\ref{plantageback}) are claimed without proof. 
\end{Remark}

\subsection{Preliminaries}
\label{startPre} 

\begin{Lemma}
 \label{filtrer}
For any solution $U:=(u,v)$ to \eqref{divgen1}-\eqref{MBgen2} with  
$U_{init} := (u_{init} ,v_{init}) \in L_{div}$ as initial data given 
by Theorem \ref{th:weak}, the part ${\bf u  } := Pu$ solves for 
$ x  \in  \R^3$,
\begin{eqnarray}
\label{selec1} 
( \dt + B ) {\bf u  }  &=&  P (A {\bf u  }) + P g , \\   
\label{selec2} 
{\bf u  }|_{t=0}   &=& Pu_{init} ,  
\end{eqnarray}
where 
\begin{eqnarray}
A(t,x) &:=& (\kappa^{-1} \cdot l) F_1(x,\overline{v}), \label{defA} \\  
g(t,x) &:=& (\kappa^{-1} \cdot l)   
F (x,\overline{v}, (Id-P) (\kappa^{-1} \cdot l) \overline{v}), 
\label{defg1} \\   
&=&   (\kappa^{-1} \cdot l)   F_0 (x,\overline{v}) 
+ (\kappa^{-1} \cdot l) 
F_1 (x,\overline{v}) (Id-P) (\kappa^{-1} \cdot l) \overline{v}). 
\label{defg2}
\end{eqnarray}
\end{Lemma}
\begin{proof}
First, apply the projector $P$ to the system \eqref{MBgen}, 
observing that  $P$  commutes with both $\partial_t$ and $B$.
Then, split $F$ according to \eqref{affine}, 
split $u$ into $u = {\bf u  } + (Id-P)u$ 
and finally use the constraint \eqref{divgen1}.
\end{proof}
The projectors $P_i$, $i=1,2$, defined on $L^2(\R^3,\R^3)$, 
extend to $L^p(\R^3,\R^3)$ (this result extends the classical one 
by Calder\'on and Zygmund \cite{CZ52} on singular integrals, 
in the spirit of the extension by Judovi{\v{c}} \cite{yudo}):
\begin{Lemma}[Starynkevitch \cite{starynkevitchPhD}, Lemma 3.13]
\label{PonLp}
Under assumption \eqref{hyp:OmegaBded}, the projectors $P_i$, $i=1,2$, 
extend to $L^p(\R^3,\R^3)$  and for all $p_0>1$, there exists $C>0$ 
such that  for all $p\in  \lbrack p_0,\infty)$, their norm   
from $L^p(\R^3,\R^3)$ into itself are less than $Cp$.
\end{Lemma}
We deduce estimates for the right-hand side of \eqref{selec1}:
\begin{Lemma}
 \label{plantage}
As in Theorem \ref{unique}, assume \eqref{affine}-\eqref{maj2} 
and \eqref{hyp:OmegaBded}. 
Let $U_{init} := (u_{init} ,v_{init}) \in L_{div}$, and 
let $U:=(u,v)$ be any solution to \eqref{divgen1}-\eqref{MBgen2} 
with $U_{init} := (u_{init} ,v_{init}) \in L_{div}$ as initial data 
given by Theorem \ref{th:weak}. 
The following holds true for $A$ and $g$ given 
by \eqref{defA}-\eqref{defg2}:
\begin{eqnarray}
\label{plantage1} A \in L^\infty_{loc}((0,\infty),L^\infty(\R^3)), \\ 
\label{plantage2} A \in C ( [0,\infty), L^2 (\R^3 )), \\ 
\label{plantage3} 
\partial_t A \in L^\infty_{loc}((0,\infty),L^2(\R^3 )), \\ 
\label{plantage4} 
g  \in \cap_{1 \le p < \infty} C ( [0,\infty), L^p (\R^3 )), \\ 
\label{plantage5} \partial_t g \in \cap_{1 \le q \le 2} 
L^\infty_{loc}( (0,\infty), L^p (\R^3 )). 
\end{eqnarray}
\end{Lemma}
\begin{proof}
For all $t,t' \geqslant 0$, there holds
\begin{eqnarray}
\| A(t) \|_{L^\infty(\R^3 )} 
\leqslant \| \kappa^{-1} \cdot l \|_{L^\infty(\R^3 )}  
C_F (\|v_{init}\|_{L^\infty}   e^{Kt} ), \\  
\| A(t) - A(t') \|_{L^2   (\R^3 )} 
\leqslant \| \kappa^{-1} \cdot l \|_{L^\infty(\R^3 )}  
C_F (\|v_{init}\|_{L^\infty} e^{Kt}) \|v(t)-v(t')\|_{L^2(\Omega)}, \\ 
\| \partial_t A (t)\|_{L^2(\R^3 )} 
\leqslant  \| \kappa^{-1} \cdot l \|_{L^\infty(\R^3 )}  
C_F (\|v_{init}\|_{L^\infty}e^{Kt}) \|\partial_t v(t)\|_{L^2(\Omega)},
\end{eqnarray}
what yields  estimates \eqref{plantage1}-\eqref{plantage3}. 

Since $v \in C([0,\infty),L^2(\Omega)) \cap 
L^\infty_{loc}((0,\infty), L^\infty(\Omega))$, 
we have $v \in \cap_{p \geqslant 1} C([0,\infty),L^p(\Omega ))$ 
-- using the boundedness of $\Omega$ for $p<2$, and  
by interpolation for $p>2$. 
Lemma \ref{PonLp} then yields \eqref{plantage4}. 
Next, using that 
$| F_i (x,v) | \leqslant C_F (\|v_{init}\|_{L^\infty} e^{Kt}) |v|$ 
for $i=0,1$, we infer from \eqref{MBgen2} that $\partial_t v 
\in  \cap_{1 \le q \le 2} L^\infty_{loc}((0,\infty),L^p(\R^3))$. 
Since
\begin{eqnarray*}
\partial_t g (t,x) =   (\kappa^{-1} \cdot l)  \{ \partial_v F_0 (x,\overline{v}) \cdot \partial_t \overline{v} +    F_1 (x,\overline{v})  (Id-P)  (\kappa^{-1} \cdot l) \cdot \partial_t  \overline{v} +
(\partial_v F_1 (x,\overline{v}) \cdot \partial_t  \overline{v})  \cdot (Id-P)    (\kappa^{-1} \cdot l)   \overline{v}      \},
\end{eqnarray*}
thanks to Lemma \ref{PonLp}, we finally get \eqref{plantage5}.
\end{proof}
Also, for a given $v$, the ``fields part" $u$ is in fact uniquely determined: 
\begin{Lemma}
 \label{keskonepropre}
Let $A \in  L^\infty_{loc} ((0,\infty), L^\infty (\R^3 ))$ and 
$g  \in L^1_{loc} ((0,\infty), L^2 (\R^3 ))$.
For any $u_{init} \in L^2 (\R^3 )$, there exists only one solution 
${\bf u} \in C ( [0,\infty), L^2 (\R^3 ))$ to \eqref{selec1}-\eqref{selec2} with $P u_{init}$ as initial data. 
Furthermore, it satisfies ${\bf u} =P{\bf u}$.
\end{Lemma}
\begin{proof}
Existence is given by Lemma \ref{filtrer}. To prove uniqueness, consider two solutions  ${\bf u}_1 $ and ${\bf u}_2$ in 
$C([0,\infty), L^2(\R^3))$ to \eqref{selec1}-\eqref{selec2}, 
and $T>0$. Then 
\begin{eqnarray*}
\forall t\ge0, \quad ({\bf u}_1 - {\bf u}_2) (t) = 
\int_0^t e^{i(t-s)B} PA ({\bf u}_1 - {\bf u}_2) (s) ds ,
\end{eqnarray*}
so that, using  \eqref{plantage1}, for $t  \in [0,T]$, 
\begin{eqnarray*}
\| ({\bf u}_1 - {\bf u}_2)(t) \|_{L^2(\R^3)}  \leq C(\kappa,T) 
\int_0^t \| ({\bf u}_1 - {\bf u}_2) (s) \|_{L^2(\R^3)} ds . 
\end{eqnarray*}
Hence, by Gronwall's Lemma, ${\bf u}_1 = {\bf u}_2$ on $[0,T]$, 
for any $T>0$. Thus, there is only one solution ${\bf u}$. 
Finally, in the same way, ${\bf u}-P{\bf u}$ simply satisfies:
\begin{eqnarray*}
\forall t\ge0, \quad (\dt+B) ({\bf u} - P{\bf u}) (t) = 0 ,
\end{eqnarray*}
so that ${\bf u} = P{\bf u}$.
\end{proof}

\subsection{Technical interlude $2$: Fourier analysis}
\label{prelim2} 
 
We recall  the existence of a smooth dyadic partition of unity:  
there exist two radial bump functions 
$\chi $ and $\phi$  valued in the interval $\lbrack 0,1 \rbrack$, supported respectively in the ball  $B(0,4/3) := \{ | \xi | < 4/3 \}$ and in the annulus $ C(3/4,8/3) := \{ 3/4 < | \xi | < 8/3 \}$, 
such that 
\begin{eqnarray*}
\forall \xi \in \R^3 , \quad \chi (\xi) 
+ \sum_{ j \geqslant 0 } \phi  (2^{-j} \xi ) = 1, \qquad\qquad 
\forall \xi \in \R^3 \setminus \{0\} ,\quad   
\sum_{ j \in \Z} \phi  (2^{-j} \xi ) = 1 , \\ 
| j-j' |  \geqslant 2 \Rightarrow 
\text{ supp }  \phi  (2^{-j}  \cdot  ) \cap   
\text{ supp } \phi  (2^{-j'} \cdot ) =  \emptyset , \qquad
j \geqslant 1  \Rightarrow   
\text{ supp }  \chi  (2^{-j}  \cdot  ) \cap   
\text{ supp } \phi  (2^{-j} \cdot ) =  \emptyset .
\end{eqnarray*}
The Fourier transform $\mathcal{F}$ is defined on the space of integrable functions $f \in L^1 (\R^3)$ by 
$(\mathcal{F} f) (\xi) := \int_{\R^3} e^{-2i\pi x\cdot\xi} f(x) dx$, 
and extended to an automorphism of the space $\mathcal{S}'(\R^3)$  
of  tempered distributions, which is the dual of the Schwartz space 
$\mathcal{S}(\R^3)$ of rapidly decreasing functions.
  
The so-called dyadic blocks $\Delta_j $ correspond to the Fourier 
multipliers 
$ \Delta_j :=  \phi (2^{-j} D) $, that is
\begin{eqnarray*}
\Delta_j u (x):=  2^{3j} \int_{\R^3}  \tilde{h} (2^{j}y) u(x-y) dy  
\quad \text{ for }  j  \geqslant 0 , \quad \text{ where }   
\tilde{h} :=  \mathcal{F}^{-1}  \phi .
\end{eqnarray*}
We also introduce $S_0  :=  \chi (D) $, that is
\begin{eqnarray*}
S_0 u (x):=   \int_{\R^3  }  {h} (y) u(x-y) dy  , \quad 
\text{ where }   {h} :=  \mathcal{F}^{-1}  \chi .
\end{eqnarray*}
We will use the inhomogeneous  Littlewood-Paley decomposition 
$ Id =   S_{-1} + \sum_{ j \in \N } \Delta_j   $, which holds 
in  the space of tempered distributions $\mathcal{S}' (\R^3)$, 
and the homogeneous  Littlewood-Paley decomposition 
$ Id =  \sum_{ j \in \Z } \Delta_j   $, which holds in  
$\mathcal{S}'_h (\R^3)$, the space of tempered distributions $u$ 
such that $\lim_{j \rightarrow - \infty} 
\| \sum_{ k \leqslant j } \Delta_k  u  \|_{L^\infty(\R^3 )} = 0$.
    
We now recall the definition of the   inhomogeneous 
(respectively homogeneous)  Besov spaces  
$ {B}^\lambda_{p,q}  $ (resp. $ \dot{B}^\lambda_{p,q}  $) on  $\R^3$ 
which are, for  $\lambda  \in \R$ (the smoothness index), 
$p,q  \in \lbrack 1,+  \infty \rbrack$ 
(respectively the integral-exponent and the sum-exponent), 
the spaces of tempered distributions $u$  in  $\mathcal{S}'(\R^3)$  
(resp. $\mathcal{S}'_h (\R^3)$) such that
\begin{eqnarray*}
\| f \|_{{B}^\lambda_{p,q} (\R^3)} := \| S_0  f \|_{ L^p (\R^3)}  
+ \| (2^{j \lambda} \| \Delta_j  f \|_{L^p(\R^3)})_{j}\|_{l^q(\N)}
\quad ( \text{resp. } \| f \|_{ \dot{B}^\lambda_{p,q}(\R^3)} :=  
\| (2^{j \lambda} \| \Delta_j f \|_{L^p(\R^3)})_{j } \|_{l^q(\Z)} )
\end{eqnarray*}
is finite. 
These Banach spaces do not depend on the choice of the dyadic 
partition above (cf. for instance the book \cite{bcd}).

\subsection{Dispersion}
\label{Disper} 

Propagation of  smoothness or singularities for solutions to 
hyperbolic Cauchy problems, such as
\begin{align}
\label{R1} 
L u := (\partial_t + B  )u=  f, \text{ with }   u  |_{t=0} = u_{init} ,
\end{align}
obeys the laws of geometrical optics. Let us refer here to the 
survey  \cite{Garding} by G{\aa}rding 
for an introduction to the subject. 
The characteristic variety of the operator $ L$  is defined as
\begin{eqnarray*}
\text{Char} (L) &:=& \{  (t,x, \tau,\xi) \in \R  \times \R^3 
\times \C  \times (\R^3 \setminus \{0\}) 
\mid \det L (t,x,\tau,\xi ) = 0\} ,
\end{eqnarray*}
where $L (t,x, \tau,\xi ) $ denotes the (principal) symbol  of 
the operator $L$, which is the $6 \times 6$ matrix 
\begin{eqnarray*}
L (t,x,\tau,\xi) \equiv L (x,\tau,\xi) &:=& \tau Id +  B (x,\xi), 
\quad \text{ where } 
B (x,\xi ):=
\begin{bmatrix}
0 & \kappa_1(x)^{-1} \xi \wedge \cdot 
  \\  \kappa_2(x)^{-1}  \xi \wedge \cdot  & 0 
\end{bmatrix}.
 \end{eqnarray*}
For  all $(x, \xi) \in \R^3  \times (\R^3 \setminus \{0\})$, 
the matrix $ B (x,\xi )$ admits three eigenvalues: 
$\lambda_\pm (x,\xi)  :=  \pm (\kappa_1 \kappa_2(x))^{-1/2} |\xi |$ 
and $0$, each one with multiplicity $2$.
The eigenspace associated with the eigenvalue $0$ is precisely 
ran $(Id - P)$. We introduce
$$ 
P_\pm (x, \xi) := \frac{1}{2\pi i} 
\oint_{| z- \lambda_\pm (x,\xi) |=r} L (x, -z,\xi )^{-1} dz, 
$$
where $r$ is chosen small enough for $\lambda_\pm (x,\xi)$ being 
the only eigenvalue inside the circle of integration. The matrix 
$P_\pm(x,\xi)$ is the spectral projection associated with the 
eigenvalue $\lambda_\pm (x, \xi)$, that is the projection onto the 
kernel of $L \big(x,-\lambda_\pm(x,\xi),\xi \big)$ along its range. 
These spectral projections are homogeneous of degree $0$ w.r.t. $\xi$, 
and the associated pseudo-differential operators $P_\pm$ satisfy  
$ P_+ +  P_-  = Id - P $. 
In addition, $ Id-P$, $P_+ $ and $P_-$ are orthogonal projectors 
(in the weighted $L^2$ space introduced in Section~\ref{abstract})  
commuting with $B$, and acting on Besov spaces.

Now the point is that  considering the Cauchy problem \eqref{R1} for 
solutions $u$ satisfying $ (Id-P)u=0$, we select the branch of the 
characteristic variety which are curved, what generates dispersion. 
We shall need the following indices $p_1$, $r_1$, $q_1$, $s_1$, 
$\mu$, $\sigma$ and $\rho$:
\begin{eqnarray}
& \label{def:indices1} p_1 \in [2,\infty) \quad \text{and} \quad 
1/r_1 + 1/p_1 = 1/2 ; \\
& q_1 \in (1,2] \quad \text{and} \quad 1/s_1 + 1/q_1 = 3/2 ; \\
& \label{def:indices3} \mu \in \R, \quad \sigma := \mu - 1 + 2/p_1 \quad \text{and} \quad 
\rho := \mu - 1 + 2/q_1 . 
\end{eqnarray}
\begin{Proposition}
 \label{StrichartzProppalim}
Let $p_1$, $r_1$, $q_1$, $s_1$, $\mu$, $\sigma$ and $\rho$ be given 
by  \eqref{def:indices1}-\eqref{def:indices3}. 
Under assumption \eqref{hyp:kappa}, there is a non-decreasing function  
$C:(0,\infty) \rightarrow (0,\infty)$ such that, for any $T>0$, 
for any initial data $u_{init}$ in $\dot{H}^\mu$ such that 
$(Id-P)u_{init} =0$, and for any source term  $f$ in 
$L^{s_1} ((0,T), \dot{B}^\rho_{q_1 ,2} (\R^3) )$ such that 
$ (Id-P)f=0$, any (weak) solution $u$ to the Cauchy problem 
\eqref{R1} belongs to 
$L^{r_1} ((0,T), \dot{B}^\sigma_{p_1 ,2} (\R^3) )$ and satisfies 
$u=Pu$, as well as 
\begin{eqnarray} 
\label{StrichartzProppalimesti}
\| Pu \|_{L^{r_1}\left((0,T),\dot{B}^\sigma_{p_1,2}(\R^3)\right)} 
\leq C(T) \Big( \| Pu_{init} \|_{\dot{H}^\mu} 
+ \| Pf \|_{ L^{s_1}\left((0,T),\dot{B}^\rho_{q_1 ,2}(\R^3)\right)} 
\Big) .
\end{eqnarray}
In the case $p_1 = 2$ (hence $r_1 =\infty$, $\sigma=\mu$), 
the function $u=Pu$ is even in $C ([0,T], \dot{H}^\mu  (\R^3) )$.
\end{Proposition}
%

The result of Proposition \ref{StrichartzProppalim} is false for  
$r_1=2$, $p_1=\infty$, $s_1=1$, $q_1=2$ and $\mu=\rho=1$, $\sigma=0$. 
However, it is true when truncating frequencies. 
We use, for $\lambda>0$, the low frequency cut-off operator 
$S^\lambda$, which is the Fourier multiplier with symbol
$\chi_\lambda:=\chi(\cdot/\lambda)$, where the cut-off function 
$\chi \in {\mathcal C}^\infty_c(\R^d,[0,1])$ takes value 
$1$ when $|\xi| \leq 1/2$, and $0$ when $|\xi| \geqslant 1$.
Then, we have:
\begin{Proposition}
 \label{StrichartzProp}
Under assumption \eqref{hyp:kappa}, there is a non-decreasing function 
$C:(0,\infty) \rightarrow (0,\infty)$ such that, for all $\lambda, T>0$ 
and for any $u \in \mathcal{C}([0,T),H^1(\R^3))$ solution to \eqref{R1}, 
%
%
\begin{eqnarray} 
\label{StrichartzEsti}
\|S^\lambda Pu\|_{L^2((0,T),L^\infty(\R^3))} \leq C(T)
\sqrt{\ln(1+\lambda T)} \left( \| \partial_{x} Pu_{init} \|_{L^2(\R^3)} 
+ \| \partial_{x} Pf \|_{L^1((0,T),L^2(\R^3))} \right).
\end{eqnarray}
\end{Proposition}
Estimate \eqref{StrichartzEsti} is proved in the case where 
the operator $P$ is $P:= -\Delta$ in \cite{jmr} Proposition $6.3$. 
Even in this case Estimate \eqref{StrichartzEsti} without the 
cut-off $S^\lambda$ is false \cite{km}, \cite{lindblad}. 
Proposition \ref{StrichartzProp} extends \cite{jmr}'s result to 
(smooth) variable coefficients.

Let us mention that one can deduce from these results some similar 
estimates for the $\R^N$-valued solutions $u$ of  wave equations 
of the form:
\begin{eqnarray*} 
\label{onde1} 
(\dt^2+ A) u = f \text{ in } \R^3 ,
\quad  \dt^\nu u |_{t=0} = u_\nu \quad \text{ for } \nu=0,1 ,
\end{eqnarray*}
where 
$A:= - a(x)  \Delta + \sum_{j=1}^n B_j (x)  \partial_j + C(x) $, when $a-1 \in \mathcal{C}^{\infty}_K(\R^3 )$, $a(x) \geqslant c_0 > 0$ and 
the $B_j$ and $C$ are in $ \mathcal{C}^{\infty}_K(\R^3, \mathcal{M}_{N \times N}(\R) )$.
These equations stand for the propagation of waves in an inhomogeneous isotropic compact medium $K \subset \R^3$ surrounded by vacuum.

To prove Proposition \ref{StrichartzProppalim} and  Proposition \ref{StrichartzProp} we use the Lax method, that is an explicit representation of the solution which allows 
to take advantage of oscillations \emph{via} the method of stationary phase, for each dyadic block, 
to get a pointwise dispersive estimate. The final step relies on the TT* argument and the summation over the dyadic blocks.

This kind of strategy is now very classical and we refer here to the book \cite{bcd} by Bahouri, Chemin and Danchin for a larger overview of its use and of its consequences. However we did not find  Proposition \ref{StrichartzProp} in the literature so that we now detail a little bit its proof. 

\begin{proof}[Proof of Proposition \ref{StrichartzProppalim}] 
Let us first remark that it is sufficient to prove 
Estimate~\eqref{StrichartzProppalimesti} for smooth data 
(by the usual regularization process) and locally in time. 
More precisely, it suffices to prove that there is a constant 
$C>0$ and $T_1 >0$ such that, for all $\lambda >0$ and for all 
$u \in \mathcal{C}([0,T_1],H^1(\R^3))$ solution to \eqref{R1},  
\begin{eqnarray} 
 \label{StrichartzProppalimbis}
\| Pu \|_{L^{r_1}\left((0,T_1),\dot{B}^\sigma_{p_1 ,2}(\R^3)\right)} 
\leq C \Big( \| Pu_{init} \|_{\dot{H}^\mu} 
+ \| Pf \|_{L^{s_1}\left((0,T_1),\dot{B}^\rho_{q_1,2}(\R^3)\right)} 
\Big) .
\end{eqnarray}
Indeed, apply several times Estimate~\eqref{StrichartzProppalimbis} 
on time intervals of the form $(k T_1, (k+1) T_1)$, with $k$ an 
integer ranging from zero to the integral part $K$ of $T/T_1$ 
(plus the interval $(KT_1,T)$): on the right hand side, 
$\| Pu(kT_1) \|_{\dot{H}^\mu}$ is estimated by 
$C \Big( \| Pu((k-1)T_1) \|_{\dot{H}^\mu} + 
\| Pf \|_{L^{s_1}\left(((k-1)T_1,kT_1),
\dot{B}^\rho_{q_1,2}(\R^3)\right)} \Big)$. 
Summing up gives \eqref{StrichartzProppalimesti}.

Now, consider the operator $S(t) := e^{-tB} P$. 
It admits a parametrix,  and thus is given by a sum 
$S(t) = I_+(t) + I_-(t)$ of operators on the dispersive eigenspaces, 
which are Fourier Integral Operators: for any smooth function $u(x)$, 
\begin{eqnarray*} 
(I_\pm (t) u)(x) =   \int_{\R^3 \times  \R^3 } 
e^{i \Big( \Psi (t,x,\xi) - 2\pi y\cdot\xi  \Big)} 
a (t,x,\xi) u (y) d\xi  dy ,
\end{eqnarray*}
and we drop the subscript $\pm$ in the sequel. 
The phase $ \Psi (t,x,\xi) $ is real, positively homogeneous 
of degree one in $\xi$, $\mathcal{C}^\infty$ for $\xi \neq 0$, 
and satisfies the eikonal equation:
\begin{eqnarray} 
 \label{eik1}
\dt  \Psi (t,x,\xi) = \pm (\kappa_1 \kappa_2(x))^{-1/2} ~ |  \Psi'_x (t,x,\xi) |, 
\end{eqnarray}
with
\begin{equation} 
 \label{eik2}
\Psi |_{t=0} (x,\xi) =  2\pi x\cdot\xi .
\end{equation} 
The amplitude $a$ is in H{\"o}rmander's class $S^{0}$, 
and admits an asymptotic expansion whose successive orders satisfy a sequence of linear hyperbolic 
equations. Such a method was initiated by Lax in his pioneering paper \cite{lax}. 
Because of the caustic phenomenon, the lifespan of smooth solutions to Equation \eqref{eik1} is limited.  
However, since $ \Psi $ is  homogeneous in $\xi$ and the set $K \times S^2$ is compact, there exists 
$T_1 >0$ such that the solution to \eqref{eik1}-\eqref{eik2} remains smooth on $\lbrack -T_1,T_1 \rbrack$.
Let us mention here that  Ludwig \cite{ludwig} succeeded in extending Lax' analysis into a global-in-time result.  The arguments have been refined thanks to H{\"o}rmander's theory of Fourier Integral Operators \cite{fio1}, \cite{fio2}. But we use the parametrices (and the solution operators $I(t)$) only locally in 
time and we refer for their  construction to the work of Chazarain  \cite{chazarain}, Nirenberg and Treves \cite{nt1} and  \cite{nt2}, Kumano-go \cite{kumano} and Brenner \cite{brenner}. 
In particular we refer to the last one for the following precious informations about the phase 
(\cite{brenner} Lemma 2.1): there exist $c, C>0$ such that 
\begin{enumerate}[(i).]
\item \label{i} $c |\xi | \leqslant |\Psi'_x| \leqslant C |\xi|$ 
on $\lbrack-T_1,T_1\rbrack \times \R^3 \times (\R^3 \setminus \{0\})$;
\item \label{ii} $c Id \leqslant \pm\Psi''_{x \xi} \leqslant  C Id$, 
$\Psi''_{x\xi}$  being real symmetric, 
on $\lbrack-T_1,T_1\rbrack \times \R^3 \times (\R^3 \setminus \{0\})$; 
\item $\Psi''_{\xi \xi}$ is semi-definite with rank $2$ for 
$|\xi| \neq 0 $, $t \neq 0$; 
and for $|\xi| = 1 $, $x \in K$, there is a constant $c_0 > 0$ 
such that the moduli of the non-zero eigenvalues of $\Psi''_{\xi \xi}$ 
are bounded from below by $c_0 |t|$; 
\item for $x \notin K$, the above results are consequence of the 
exact formula: $\Psi (t,x,\xi) =  
2\pi x\cdot\xi \pm 2\pi (\kappa_1 \kappa_2(x))^{-1/2} t | \xi |$.
\end{enumerate} 
Note that these results imply that  the kernel of $I(t)$ is 
a Lagrangian distribution.

We use the $TT^\star$ method for the frequency localized operators 
$$
T_j(t) := \Delta_j I(t), \, j\in\Z.
$$
The composed operator $T_j(t)T_j(t')^\star$ is then 
$$
T_j(t)T_j(t')^\star = \Delta_j I(t-t') \Delta_j.
$$
Here is the $TT^\star$ result.
\begin{Lemma} 
\label{Lemmatt}
There exist $0< c < C$ such that for all $j \in \Z$, 
$u \in L^2 (\R^3)$, $p,r  \in [1,\infty]$ 
and $f \in L^{r'}((0,T),L^{p'}(\R^3 ))$,
\begin{eqnarray} 
\label{tt1}
\| T_j(t) u \|_{L^{r} ((0,T), L^{p} (\R^3 ))} \leqslant 
\| u \|_{L^2 (\R^3)} \sup_{g \in \mathcal{B}_{r,p }^j } 
(b_j (g,g) )^\frac{1}{2} ,
\\  \label{tt2} \text{and} \quad
\left\| \int_0^T  T_j(t) T_j(t')^\star f(t') dt'  
\right\|_{L^{r}((0,T),L^{p}(\R^3))}
= \sup_{g \in \mathcal{B}_{r,p }^j }  |b_j (f,g)|  ,
\end{eqnarray}
with
\begin{eqnarray*} 
b_j (f,g) :=  \int_{(0,T) \times (0,T)} 
\langle T_j(t) T_j(t')^\star f(t') , g(t) \rangle_{L^2(\R^3)} dt dt' ,
\end{eqnarray*}
$\mathcal{B}_{r,p }^j$ the space of functions $g(t,x)$ in 
$\mathcal{B}_{r,p }$  whose Fourier transform is supported in 
$c  2^{j} \leqslant |\xi| \leqslant  C  2^{j}$, 
and $\mathcal{B}_{r,p }$ the space of smooth functions $g(t,x)$ 
satisfying 
$\| g \|_{L^{r'}\left((0,T),L^{p'}(\R^3 )\right)} \leqslant 1$.
\end{Lemma} 
\begin{proof}
Begin with
\begin{eqnarray*} 
\| T_j (t) u   \|_{L^{r} ((0,T), L^{p} (\R^3 ))} =  \sup_{g \in \mathcal{B}_{r,p } } 
\left| \int_{(0,T)\times \R^3  } T_j (t) u(x) \cdot g(t,x) dt dx  \right| .
\end{eqnarray*}
Then, use the Plancherel identity plus the properties of the support of the Fourier transform of 
$T_j(t) u$ to get 
\begin{eqnarray*} 
\| T_j (t) u   \|_{L^{r} ((0,T), L^{p} (\R^3 ))} = \sup_{g \in \mathcal{B}_{r,p }^j }  
\left| \int_{(0,T)\times \R^3  } T_j (t) u \cdot g dt dx \right| .
\end{eqnarray*}
Using Fubini's principle, transposition  and Cauchy-Schwarz inequality, observe that 
\begin{eqnarray} 
 \label{dual2}
\| T_j (t)   u \|_{L^2((0,T),L^\infty(\R^3))}  \leq  \|   u \|_{L^2  (\R^3)} 
\sup_{g \in \mathcal{B}_{r,p }^j } \left\| \int_0^T T_j(t)^\star g  dt \right\|_{L^2  (\R^3)} .
\end{eqnarray}
Using again Fubini's principle and transposition, get
\begin{eqnarray} 
\label{inter}
\left\| \int_0^T T_j(t)^\star g(t) dt \right\|_{L^2(\R^3)}^2 
&= &  b_j (g,g) .
\end{eqnarray}
This leads to \eqref{tt1}.
Moreover,
\begin{eqnarray*} 
\left\| \int_0^T T_j(t) T_j(t')^\star f(t') dt'  \right\|_{L^{r} ((0,T), L^{p} (\R^3 ))}
&=& \sup_{g \in \mathcal{B}_{r,p }^j }  
\left| \int_{(0,T)\times \R^3} \Big(\int_0^T T_j(t) T_j(t')^\star f(t') dt' \Big) \cdot g(t) dt dx \right|
\\ &=& \sup_{g \in \mathcal{B}_{r,p }^j }  |b_j (f,g)|  .
\end{eqnarray*}
\end{proof}
Now, we need to estimate $b_j(f,g)$. We start with a  pointwise estimate.
\begin{Lemma} 
There exists $C>0$ such that for all $j \in \Z$, $(t,t')  \in (0,T) \times  (0,T) $ 
and $u \in L^1 (\R^3)$, 
\begin{eqnarray*} 
\| T_j(t) (T_j)(t')^\star u   \|_{L^{\infty} (\R^3 )} \leqslant C 2^{3j}   (1 + 2^j |t-t'| )^{-1}
 \|  u   \|_{L^1 (\R^3 )}  .
\end{eqnarray*}
\end{Lemma} 
\begin{proof}
Since $\Delta_j$ is a bounded operator on $L^p(\R^3)$ 
for all $p\in[1,\infty]$, it is sufficient 
to prove that for all $t \in (-T,T)$ and $u \in L^1(\R^3)$, 
\begin{equation} \label{estimIDelta}
\| I(t) \Delta_j u   \|_{L^{\infty} (\R^3 )} \leqslant C 2^{3j} 
(1 + 2^j |t| )^{-1} \|  u   \|_{L^1 (\R^3 )}  .
\end{equation}
Writing down $I(t) \Delta_j u$, we get, for all $x\in\R^3$:
\begin{equation*}
\begin{split}
(I(t) \Delta_j u) (x) 
& = \int_{\R^3 \times  \R^3 } 
e^{i \Big( \Psi (t,x,\xi) - 2\pi y\cdot\xi  \Big)}  
a (t,x,\xi) \varphi(2^{-j}D) u (y) d\xi  dy \\
& = \int_{\R^3 } e^{i \Psi (t,x,\xi)} a (t,x,\xi) 
\varphi(2^{-j}\xi) \hat{u} (\xi) d\xi \\
& = 2^{3j} \int_{\R^3 } e^{i 2^j \Psi (t,x,\eta)} a (t,x,2^j\eta) 
\varphi(\eta) \hat{u} (2^j\eta) d\eta ,
\end{split}
\end{equation*}
so that
$$
|(I(t) \Delta_j u) (x)| \le 
2^{3j} \sup_{y\in\R^3} \left| 
\int_{\R^3 } e^{i 2^j (\Psi (t,x,\eta) - y\cdot\eta)} 
a (t,x,2^j\eta) \varphi(\eta) d\eta \right| \, \|u\|_{L^1(\R^3)} .
$$
To get \eqref{estimIDelta}, simply apply  the following lemma 
of stationary phase.
\begin{Lemma}  [Littman \cite{littman}]
Let $\Psi (\xi)$ be a real function $\mathcal{C}^\infty$ such that 
the rank of its Hessian matrix $ \Psi''_{\xi \xi}$ is at least $\rho$ 
and let $v(\xi) $ be a function supported in a ring. 
Then there exists $M \in \N$ and $C > 0$ (which depends only on 
a finite number of derivatives of $ \Psi$, of a lower bound of 
the maximum of the abolute values of the minors of order $\rho$ of 
$\Psi''_{\xi \xi}$, on supp $v$) such that, for all $\Lambda\in\R$, 
\begin{eqnarray*} 
\| \mathcal{F}^{-1} (e^{i\Lambda  \Psi} v ) \|_{L^\infty(\R^3))} \leq 
C (1 + |\Lambda| )^{-\frac{\rho}{2}} \cdot \sum_{|\alpha| \leqslant M} 
\| D^\alpha v  \|_{L^1(\R^3)}  .
\end{eqnarray*}
\end{Lemma} 
To use this lemma, distinguish between short times, for which 
the eikonal equation \eqref{eik1} implies that the phase $\Psi$ 
admits the expansion 
\begin{eqnarray*} 
 \Psi (t,x,\xi) =  2\pi  \Big( x\cdot\xi \pm t (\kappa_1\kappa_2(x))^{-1/2} | \xi |  \Big) + O(t^2 ),
\end{eqnarray*}
(thus $\rho=2$, $\Lambda=2^j(\kappa_1\kappa_2(x))^{-1/2}t$ is 
suitable), and subsequent times, for which Estimate~\eqref{ii} 
on the phase gives $\rho=2$ (and $\Lambda=2^j$). This yields  
\begin{eqnarray} 
\label{sphase}
\sup_{y\in\R^3} \left| 
\int_{\R^3 } e^{i 2^j (\Psi (t,x,\eta) - y\cdot\eta)} 
a (t,x,2^j\eta) \varphi(\eta) d\eta \right| 
\leqslant C (1 + 2^{j} t )^{-1} .
\end{eqnarray}
\end{proof}
Since $T_j(t) T_j(t')^\star$ is also bounded on $L^2(\R^3)$, 
from the above result and the Riesz-Thorin Interpolation Theorem, 
we infer the following.
\begin{Lemma} 
There exists $C>0$ such that for all $j \in \Z$, $(t,t') \in (0,T) 
\times (0,T)$, $p\in\lbrack2,\infty\rbrack$ and $u \in L^{p'}(\R^3)$,
\begin{eqnarray} 
\label{thorin}
\| T_j(t) T_j(t')^\star u \|_{L^{p} (\R^3 )} \leqslant 
C 2^{3j(1-2/p)} (1 + 2^{j}|t-t'|)^{-(1-2/p)}\| u \|_{L^{p'}(\R^3)} .
\end{eqnarray}
\end{Lemma} 
\begin{Lemma} \label{lem:estimbj} 
For any $p_1,p_2 \in \lbrack 2,\infty )$, define $r_1$ and $r_2$ 
by $1/r_1 +1/p_1 =1/2$ and $1/r_2 + 1/p_2 = 1/2$. 
Then, there exists $C>0$ such that for all $j\in\Z$, 
$f \in L^{r_1'} ((0,T), L^{p_1'} (\R^3 ))$ and 
$g \in L^{r_2'} ((0,T), L^{p_2'} (\R^3 )) $, 
\begin{eqnarray} 
\label{bibi}
|b_j (f,g)| \leqslant C 2^{2j(1/r_1+1/r_2)} 
\| f \|_{L^{r_1'}\left((0,T),L^{p_1'}(\R^3)\right)}
\| g \|_{L^{r_2'}\left((0,T),L^{p_2'}(\R^3)\right)}  .
\end{eqnarray}
\end{Lemma} 
\begin{proof}
Using \eqref{thorin}, we apply  H\"older's inequality to get
\begin{eqnarray*} 
|b_j (f,g)| & \leqslant & C 2^{6j/r_1} \Big|\Big| 
\left( (1 + 2^j |\cdot|)^{-2/ r_1} \star 
\| f \|_{L^{p'_1}(\R^3)} \right) (t')  \, 
\| g(t') \|_{L^{p'_1}(\R^3)} \Big|\Big|_{L^1(0,T)} .
\end{eqnarray*}
Using again H\"older's inequality, we  get
\begin{eqnarray*} 
 |b_j (f,g)| & \leqslant &  C 2^{6j/r_1} 
\left\| \left( (1 + 2^j |\cdot | )^{-2/r_1} \star 
\| f \|_{L^{p'_1}(\R^3)} \right)(t') \right\|_{L^{r_1}(0,T)} \, 
\| g \|_{L^{r'_1}\left((0,T), L^{p'_1}(\R^3)\right)} .
\end{eqnarray*}
Thus, the Hardy-Littlewood-Sobolev inequality implies 
\begin{eqnarray} 
\label{interpo1}
|b_j (f,g)| \leqslant C 2^{4j/r_1}  
\| f \|_{L^{r'_1}\left((0,T),L^{p'_1}(\R^3)\right)} 
\| g \|_{L^{r'_1}\left((0,T), L^{p'_1}(\R^3)\right)}   .
\end{eqnarray}
Moreover, \eqref{inter} yields  
\begin{eqnarray} 
\label{inter2}
\left\| \int_0^T  T_j(t)^\star f(t)  dt \right\|_{L^2(\R^3)}  
= b_j(f,f)^{1/2} \leq  
C 2^{2j/r_1} \| f \|_{L^{r'_1}\left((0,T),L^{p'_1}(\R^3)\right)}.
\end{eqnarray}
Now, we observe that 
\begin{eqnarray*} 
b_j (f,g) =  \int_0^T \langle \int_0^T T_j(t')^\star f(t') dt' , T_j(t)^\star g(t) \rangle_{L^2(\R^3)} dt .
\end{eqnarray*}
We use the Cauchy-Schwarz and  the uniform boundedness of  the $T_j(t)^\star$, for $t$ in $(0,T)$, in $L^2(\R^3)$ to get
\begin{eqnarray} 
\label{enplus}
| b_j (f,g) | \le  b_j (f,f)^{1/2} \| g \|_{L^1\left((0,T),L^2(\R^3 )\right)} .
\end{eqnarray}
Using now 
Eq.~\eqref{inter2} yields
\begin{eqnarray} 
\label{interpo2}
| b_j (f,g) | \le C 2^{2j/r_1}  
\| f \|_{L^{r'_1}\left((0,T),L^{p'_1}(\R^3)\right)} 
\| g \|_{L^1\left((0,T),L^2(\R^3 )\right)} .
\end{eqnarray}
Interpolating between \eqref{interpo1} and \eqref{interpo2}, we find \eqref{bibi}.
\end{proof}
%

To complete the proof of Proposition \ref{StrichartzProppalim}, we need to sum over all frequencies.
Apply Duhamel's principle to get  
\begin{equation} \label{eq:duhamel}
Pu(t) = e^{-tB} Pu_{init} + \int_0^t e^{(t'-t)B} Pf(t') dt'.
\end{equation}
The first term in $\Delta_j Pu(t)$ is thus 
$$
\Delta_j I(t) u_{init} = \sum_{k\in\Z} \Delta_j I(t) \Delta_k u_{init}.
$$ 
An important fact in order to estimate $\Delta_j I(t) u_{init}$ is 
that the quasi-orthogonality of the dyadic blocks is not destroyed 
by the parametrix $I(t)$. 
Actually, according to \cite[Proposition 1.1 and Lemma 1.4]{brenner}, 
there exists $L \in \N^*$ such that  for all  $j \in \Z$,  
\begin{equation*}
\Delta_j I(t) = \sum_{| j-k | \leqslant L} \Delta_j I(t) \Delta_k 
+ R_j (t) , 
\end{equation*}
where $ R(t) := \sum_{j  \in \Z} R_j (t)$ satisfies  
\begin{eqnarray}
\label{rest} 
\| R (t) v \|_{L^{r_1} ((0,T), \dot{B}^\sigma_{p_1 ,2} (\R^3) ) } 
\leq C  \| v \|_{\dot{H}^\mu }  .
\end{eqnarray}
This shows that there is $C>0$ such that
\begin{equation*}
\| \Delta_j I(t) u_{init} \|_{L^{p_1}(\R^3)} 
\le C \| \Delta_j I(t) \Delta_j u_{init} \|_{L^{p_1}(\R^3)} 
+ \| R_j (t) u_{init} \|_{L^{p_1}(\R^3)} ,
\end{equation*}
so that 
\begin{equation*}
\begin{split}
\| I(t) u_{init} & \|_{L^{r_1}((0,T),\dot{B}^\sigma_{p_1,2}(\R^3))} \\
& \le C \left\| \| 
( 2^{j\sigma} \| \Delta_j I(t) \Delta_j u_{init} \|_{L^{p_1}(\R^3)} )_j 
\|_{l^2(\Z)} \right\|_{L^{r_1}(0,T)} 
+ \left\| \| 
( 2^{j\sigma} \| R_j (t) u_{init} \|_{L^{p_1}(\R^3)} )_j \|_{l^2(\Z)} \right\|_{L^{r_1}(0,T)} \\
& \le C \| ( 2^{j\sigma} \| \Delta_j I(t) \Delta_j u_{init} \|_{L^{r_1}((0,T),L^{p_1}(\R^3))} )_j \|_{l^2(\Z)} 
+ \| ( 2^{j\sigma} \| R_j (t) u_{init} \|_{L^{r_1}((0,T),L^{p_1}(\R^3))} )_j \|_{l^2(\Z)} \\
& \hspace{10cm} \text{by Minkowski's inequality} \\
& \le C \| 
( 2^{j(\sigma+2/r_1)} \| \Delta_j u_{init} \|_{L^2(\R^3)} )_j 
\|_{l^2(\Z)} 
+ \| ( 2^{j\sigma} \| R_j (t) u_{init} \|_{L^{r_1}((0,T),L^{p_1}(\R^3))} )_j \|_{l^2(\Z)} \\
& \hspace{10cm} \text{by } \eqref{tt1} \text{ and } \eqref{bibi} \\
& \le C \| u_{init} \|_{\dot{H}^\mu} \text{ for some new constant } C, 
\text{ using } \mu=\sigma+2/r_1 \text{ and } \eqref{rest}. 
\end{split}
\end{equation*}
The estimate for $\displaystyle \int_0^t e^{(t'-t)B} Pf(t') dt'$ 
follows the same lines, using \eqref{tt2} instead of \eqref{tt1}.

To conclude the proof of  Proposition \ref{StrichartzProppalim}, 
consider the case $p_1=2$. 
As \eqref{StrichartzProppalimesti} holds, it is easy to see that 
$u$ is continuous w.r.t. time: using a smooth approximation 
$f_k \in L^{s_1}((0,T),\dot{H}^\mu)$ of $f$, we get $Lu_k=f_k$, 
and since $s_1 \ge 1$, the usual energy estimates show that 
$(u_k)_k$ is a Cauchy sequence in $C([0,T],\dot{H}^\mu)$. 
By \eqref{StrichartzProppalimesti}, as $f_k$ tends to $f$ 
in $L^{s_1}((0,T),\dot{B}^\rho_{q_1,2})$, 
$u_k$ tends to $u$ in $C([0,T],\dot{H}^\mu)$. 
\end{proof}
\begin{proof}[Proof of Proposition \ref{StrichartzProp}]
Now, in order to prove Proposition \ref{StrichartzProp}, 
it suffices to come back to the proof of Lemma~\ref{lem:estimbj}, 
taking $p_1=\infty$ (and $r_1=2$). Using the standard Young 
inequality instead of the Hardy-Littlewood-Sobolev inequality 
leads to 
\begin{equation} \label{interpo1bis} 
| b_j (f,g) | \le C 2^{2j}  \ln (1 + 2^j T ) 
\| f \|_{L^2((0,T),L^1(\R^3 ))} \| g \|_{L^2((0,T),L^1(\R^3 ))} .
\end{equation}
In particular for any $g$ in $L^2((0,T),L^1(\R^3 )) $, this yields
\begin{equation*}
| b_j (g,g) |^{1/2} \le C 2^{j}  \sqrt{\ln (1 + 2^j T ) }
 \| g \|_{L^2((0,T),L^1(\R^3 ))} .
\end{equation*}
Then we apply Lemma \ref{Lemmatt} with  $(p,r) = (2,\infty)$ 
and use \eqref{enplus}, where we commute $f$ and $g$, to estimate 
the right hand side of \eqref{tt2}. We find that for any  
$u \in L^{2}(\R^3 )$, for any  $f \in L^{1}((0,T),L^{2}(\R^3 ))$,
\begin{eqnarray*}
\| T_j(t) u \|_{L^{2} ((0,T), L^{ \infty} (\R^3 ))} \leqslant 
 C 2^{j} \sqrt{\ln (1 + 2^j T ) } \| u \|_{L^2 (\R^3)} ,\\
\|   \int_0^T  T_j(t) T_j(t')^\star f(t') dt'  \|_{L^{2} ((0,T), L^{  \infty} (\R^3 ))}
 \le   C 2^{j} \sqrt{\ln (1 + 2^j T ) }
 \| f \|_{L^1 ((0,T),L^2 (\R^3 ))} .
\end{eqnarray*}
Then we proceed as in the proof of Proposition \ref{StrichartzProppalim} to sum over the dyadic blocks whose frequencies are below the cut-off parameter $\lambda$, with an extra factor $ \sqrt{\ln (1 + \lambda T ) }$ (and \eqref{rest} holds for 
$r_1=2$, $p_1=\infty$).
\end{proof}

\subsection{Propagation of smoothness: proof of Theorem~\ref{unique}, 
case where $\mu \in (0,1)$}
\label{proofunique1}

In the case where  $\mu \in (0,1)$, Theorem~\ref{unique} 
is a consequence of the following proposition:
\begin{Proposition} 
\label{regcurl}
Suppose that $A$ and $g$ satisfy the estimates 
\eqref{plantage1}-\eqref{plantage5}, that  $\mu \in (0,1)$ and 
$u_{init} \in L^2(\R^3,\R^6 )$ with $Pu_{init} \in H^\mu(\R^3,\R^6)$. 
Then the unique solution   ${\bf u}  \in C ([0,\infty),L^2(\R^3 ))$ 
of \eqref{selec1}-\eqref{selec2} given by Lemma \ref{keskonepropre} 
belongs to $C ( [0,\infty), H^\mu (\R^3 ))$.
\end{Proposition}
In order to prove Proposition \ref{regcurl}
 we introduce, for $T>0$, the space
\begin{eqnarray}
Y^\mu (T) := C ( [0,T], H^\mu (\R^3 )) \cap 
C^1 ([0,T], H^{\mu-1}(\R^3))  \cap L^r ((0,T), B^0_{p,2}(\R^3)),
\end{eqnarray}
where $p := 2/(1- \mu)$ and  $r := 2/\mu$. These indices belong to  
$(2,\infty)$. 
We shall also use the space $ Z^\mu (T) := Z^\mu_1(T) + Z^\mu_2(T)$, 
where
\begin{eqnarray}
Z^\mu_1 (T) &:=& 
L^1 ( (0,T), H^\mu  (\R^3 )) \cap C ( [0,T],L^2 (\R^3 )), \\ 
Z^\mu_2 (T) &:=&  \{ f \in C([0,T],L^2(\R^3)) \cap 
L^r((0,T),B^{-1}_{p,2}(\R^3)) \mid \\  
&& \qquad \qquad \partial_t  f \in  L^1 ( (0,T), H^{\mu-1}(\R^3)) 
+ L^s  ( (0,T), B^{0}_{q,2}  (\R^3 ))  \},
\end{eqnarray}
where $q = 2/(2- \mu)$ and  $s =  2/(1+\mu)$. 
These indices belong to $(1,2)$.

\begin{Lemma}  
\label{nereponplu} 
Suppose that $\mu,T \in (0,1)$, and $f \in Z^\mu (T)$.
Then there exists $C >0$ such that for any ${\bf u  }_{init} $ 
in $H^\mu(\R^3,\R^6)$ such that $P {\bf u}_{init} ={\bf u}_{init}$, 
the unique corresponding solution 
${\bf u} \in C([0,\infty),L^2(\R^3 ))$ to  
$L {\bf u  } = f$, with ${\bf u}|_{t=0} = {\bf u}_{init}$,  
belongs to $Y^{\mu} (T)$ and 
\begin{eqnarray}
\label{nereponpluesti} 
 \|{\bf u  } \|_{ Y^\mu (T )} \leqslant C (  \|{\bf u  }_{init}  \|_{ H^\mu (\R^3 )} + \|f \|_{Z^{\mu} (T)} ).
\end{eqnarray}
\end{Lemma} 
\begin{proof}  
By definition 
$f$ splits into $f= f_1 + f_2 $, with $f_i \in Z^\mu_i (T)$, and 
${\bf u  }$ into ${\bf u  }= u_0 + u_1 + u_2$, 
where the $u_i$ solve the following  hyperbolic Cauchy problems:
\begin{align*}
L u_0 & = S_0 f  \qquad \text{ with } 
\quad u_0 |_{t=0} = S_0{\bf u}_{init}, \\
L u_1 & = (Id-S_0) f_1  \quad \text{ with } 
\quad  u_1  |_{t=0} = (Id-S_0) {\bf u  }_{init} ,\\ 
L u_2 & = (Id-S_0) f_2  \quad \text{ with }  
\quad u_2  |_{t=0} = 0 .
\end{align*}
We already have by energy estimates that the $u_i$ are  in $C ( [0,\infty), L^2 (\R^3 ))$. Then one gets the estimates of the $u_i$ in $C^1 ( [0,T], H^{\mu-1} (\R^3 )) $ by using the equations. To get the estimates in $Y^\mu (T )$ it therefore only remains to get the estimates in $ L^r  ( (0,T), B^0_{p,2} (\R^3 ))$.

For $u_0$ this just follows the energy estimate thanks to Bernstein lemma. 

In order to estimate $u_1$ we apply Proposition 
\ref{StrichartzProppalim} with 
$p_1 = p$, $r_1 = r$, $q_1 = 2$, $s_1 = 1$, $\sigma=0$, $\rho = \mu$. This gives the estimate of $u_1$ in $ L^r ((0,T), B^0_{p,2}(\R^3))$ 
by the right-hand side of   \eqref{nereponpluesti}. 
  
In order to estimate $u_2$ we observe that 
$\partial_t u_2$ satisfies 
$L \partial_t u_2 =  (Id-S_0) \partial_t f_2 ,\text{ with }  
\partial_t u_2  |_{t=0} =  (Id-S_0) f_2 |_{t=0} $. 
By assumption there exists $g_a \in  L^1 ((0,T), H^{\mu-1}(\R^3 ))$ 
and $g_b  \in  L^s  ( (0,T), B^{0}_{q,2}  (\R^3 ))$ such that 
$\partial_t f_2 = g_a + g_b$. We split accordingly $\partial_t u_2$ 
into $ \partial_t u_2 = u_a + u_b$, where $u_a  $ solves  
$L u_a =  g_a ,\text{ with } u_a  |_{t=0} =  (Id-S_0) f_2 |_{t=0}$, 
and $u_b $ solves  $L u_b =  g_b ,\text{ with } u_b  |_{t=0} =  0$. 
In order to estimate $u_a$ (respectively $u_b$) we apply Proposition 
\ref{StrichartzProppalim} with $p_1 = p$, $r_1 = r$, 
$q_1 = 2$, $s_1 = 1$, $\sigma=-1$, $\rho =  \mu-1$  
(respectively with $p_1 = p$, $r_1 = r$, $q_1 = q$, $s_1 = s$, 
$\sigma=-1$, $\rho = 0$ and $\mu-1$ instead of $\mu$). 
This yields the estimate of $\partial_t u_2$ in  
$L^r ((0,T), B^{-1}_{p,2}(\R^3 ))$ by   $\|f_2 \|_{Z^{\mu}_2 (T)}$. 
As a consequence $Bu_2 = - \partial_t u_2  
+  (Id-S_0)  f_2 \in L^r  ( (0,T), B^{-1}_{p,2} (\R^3 )) $. 
Since $P u_2 =u_2$ this entails that  
$u_2 \in L^r  ( (0,T), B^{0}_{p,2} (\R^3 ))$.
  
\end{proof}

The proof of Proposition \ref{regcurl} follows by induction of the following lemma:
\begin{Lemma}  
\label{shortimes} 
Suppose that $A$ and $g$ satisfy the estimates \eqref{plantage1}-\eqref{plantage5}, and  $\mu \in (0,1)$.
Then there exists $T_1 > 0$ and $C >0$ such that for any  
$u_{init} \in L^2 (\R^3,\R^6)$ with $Pu_{init} \in H^\mu(\R^3,\R^6)$, 
the unique corresponding solution 
${\bf u} \in C([0,\infty),L^2(\R^3))$ to 
\eqref{selec1}-\eqref{selec2} given by Lemma \ref{keskonepropre} 
belongs to $Y^{\mu} (T_1)$ and 
\begin{eqnarray}
\|{\bf u  } \|_{ Y^\mu (T_1 )} \leqslant 
C ( \| P u_{init}  \|_{ H^\mu (\R^3 )} + \|g \|_{Z^{\mu} (T_1)} ).
\end{eqnarray}
\end{Lemma} 
\begin{proof}  
In order to prove Lemma \ref{shortimes} we  recall the following estimate:
\begin{Lemma}[Joly-M\'etivier-Rauch  \cite{jmr}, Lemma 5.3]
 \label{paraproduit}
There is a constant $C$, which depends only on 
$\| A \|_{L^\infty ((0,T), L^\infty (\R^3 ))}$, 
$\|\partial_t  A \|_{ L^\infty ((0,T), L^2 (\R^3 ))}$ and $\mu $, 
such that for all $T \in (0,1)$ and $u \in Y^\mu (T)$, 
$Au$ belongs to $ Z^\mu (T)$, and
\begin{eqnarray}
 \label{paraproduit-est}
 \|Au \|_{ Z^\mu (T)} \leqslant C T^{\mu/2} \|u \|_{Y^{\mu} (T)} + C  \|u \|_{L^\infty ((0,T), L^2 (\R^3 ))} .
\end{eqnarray}
\end{Lemma}
Let us warn the reader  that there is a small misprint in the right hand side in Lemma 5.3 of \cite{jmr}, which is corrected above.
We also have:
\begin{Lemma}  
\label{surg} 
If $g$ satisfies \eqref{plantage4}-\eqref{plantage5}, 
then for any $T>0$, $g$ belongs to $Z^\mu (T)$.
\end{Lemma} 
\begin{proof}  
Since $1< q < 2 < p < \infty $ there hold continuous embeddings 
$L^q \subset B^0_{q,2}$ and $L^p  \subset B^0_{p,2} \subset B^{-1}_{p,2} $, so that, using  \eqref{plantage4}-\eqref{plantage5}, we get, for any $T>0$, that $g$ is in $Z^\mu_2 (T) \subset Z^\mu (T)$.
\end{proof}
According to  Lemma \ref{paraproduit} and  Lemma \ref{surg} 
there is a constant $C$ such that for all $T  \in (0,1)$, 
if  $ {\bf u  } \in Y^\mu (T)$ then $f:=P (A {\bf u  }) + P g   
\in Z^\mu (T)$ and 
\begin{eqnarray}
 \label{ctoucon}
 \|f \|_{ Z^\mu (T)} \leqslant C T^{\mu/2} \| {\bf u  } \|_{Y^{\mu} (T)} + C  \| {\bf u  } \|_{L^\infty ((0,T), L^2 (\R^3 ))} .
\end{eqnarray}
 Therefore applying Lemma \ref{nereponplu}  and choosing $T_1$ small enough we get Lemma \ref{shortimes}.
\end{proof}

\subsection{Propagation of smoothness: proof of Theorem~\ref{unique} 
\eqref{Unique1}, case where $\mu =1$}
\label{proofunique2}

We now consider  $U_{init} $ in $\Ldiv$ with 
$\curl  u_{init,i}  \in L^2  (\R^3)$, for $i=1,2$, and we consider  
$U$  solution  to \eqref{divgen1}-\eqref{MBgen2} with $U_{init}$ as 
initial data given by Theorem \ref{th:weak}. The idea is to estimate 
$B(Pu) = P(Au) + Pg - \dt Pu$ from \eqref{selec1}.
\begin{Lemma}
\label{plantageback}
Define $A$ and $g$ by \eqref{defA}-\eqref{defg2}. 
The following holds true:
\begin{eqnarray}
\label{plantage6} 
\partial_t A \in  L^\infty_{loc} ((0,\infty), L^3 (\R^3 )), \\ 
\label{plantage7} 
\partial_t  g  \in   L^\infty_{loc}( (0,\infty), L^2 (\R^3 )).
\end{eqnarray}
\end{Lemma}
\begin{proof}
We apply Theorem \ref{unique}  in the case  $\mu = 1/2$.
This yields that $u$ belongs to $C ( [0,\infty), H^{1/2} (\R^3 )) $ 
and thus also to $ L^\infty_{loc} ((0,\infty), L^3 (\R^3 ))$. 
We then infer that $ \partial_t v \in  L^\infty_{loc} ((0,\infty), 
L^3 (\R^3 ))$ and then we get the estimates \eqref{plantage6} and 
\eqref{plantage6}.
\end{proof}
\begin{Lemma}
\label{dtf}
Define $A$ and $g$ by \eqref{defA}-\eqref{defg2}. 
If ${\bf u  } \in L^\infty_{loc}( (0,\infty), H^1 (\R^3 )) \cap 
W^{1,\infty}_{loc}( (0,\infty), L^2 (\R^3 ))$ then 
$f := P(A {\bf u  } + g )$ satisfies 
$  \partial_t  f  \in L^\infty_{loc}( (0,\infty), L^2 (\R^3 ))$, 
and for any $T>0$, there exists $C>0$, which only depends on $T$, 
$A$ and $g$, such that
\begin{eqnarray}
\label{dtfesti}
\| \partial_t  f \|_{ L^\infty ((0,T), L^2 (\R^3 )) } 
\leqslant  C(  \|  {\bf u  } \|_{ L^\infty ((0,T), H^1 (\R^3 )) } 
+ \| \partial_t  {\bf u  } \|_{ L^\infty ((0,T), L^2 (\R^3 )) } + 1 ).
\end{eqnarray}
\end{Lemma}
\begin{proof}
We have $\partial_t  f = P(\partial_t A {\bf u  } 
+ A \partial_t {\bf u  } + \partial_t g )$, so that using 
H\"older's inequality, the continuous embedding 
$ L^6 (\R^3 ) \subset H^1 (\R^3 )$ and the estimates   
\eqref{plantage6}-\eqref{plantage7}, we get  \eqref{dtfesti}.
\end{proof}
Now we observe that $\partial_t  {\bf u  } $ solves
 for $ x  \in  \R^3$,
\begin{eqnarray}
\label{selec1dt}
( \dt + B ) \partial_t {\bf u  }  &=&  \partial_t (P (A {\bf u  }) 
+ P g ), \\   
\label{selec2dt}
\partial_t {\bf u  }|_{t=0}   &=& -B u_{init} 
+  P (A|_{t=0} {\bf u_{init}  }) + P g|_{t=0}         .
\end{eqnarray}
This provides an estimate of $ \partial_t {\bf u  } $ in $L^\infty_{loc}( (0,\infty),  L^2 (\R^3 ))$, and using \eqref{selec1}-\eqref{selec2}, of  $B{\bf u  }$ in  
$L^\infty_{loc}( (0,\infty), L^2 (\R^3 ))$, hence of 
${\bf u  }$ in  $L^\infty_{loc}( (0,\infty), H^1 (\R^3 ))$.

\subsection{Uniqueness: proof of Theorem~\ref{unique} \eqref{Unique2}}
\label{proofunique3} 

Let us recall that in this section we assume that there exists 
$j \in \{1,2\}$ such that $l_{3-j} F  = 0$ and such that $F$ depends 
only on $(x,v,u_j)$. 
Let $U_{init} \in \Ldiv$ with $\curl u_{init,i} \in L^2(\R^3)$, 
for $i=1,2$. 
Let $U$ and $U'$ be two solutions to \eqref{divgen1}-\eqref{MBgen2}, 
given by Theorem \ref{th:weak}, both with $U_{init}$ as  initial data. 
The difference $\delta U := U'-U =( \delta u ,  \delta v )$ between
$U=(u,v)$ and $U' =(u' ,v')$ is solution to the following 
hyperbolic system 
\begin{equation}
\label{systDiff}
\begin{split}
M(\delta U) := ((\partial_t + B) \delta u , \partial_t \delta v )
= (  (\kappa^{-1} \cdot l) \delta F , \delta F) , 
\quad \mbox{ where } \quad
\delta F = F(x,v',u'_j)-F(x,v,u_j).
\end{split}
\end{equation}
Thanks to \eqref{affine} we have
\begin{equation}
 \label{brazil}
\delta F = F_0 (x,v',u'_j)-F_0 (x,v,u_j) 
+ (F_1 (x,v')- F_1 (x,v)) \cdot u_j +  F_1 (x,v') \cdot (u_j - u'_j) .
\end{equation}
The first and  last terms in  \eqref{brazil} are easily estimated in $L^2 ( \R^3)$ by $C_F ( \| v_0 \|_{L^\infty (  \R^3)} e^{Kt} )  \| \delta U \|_{L^2 (  \R^3)}$.
To deal with the second one we construct a $L^\infty$ approximation 
of the field $u_j$ (analogous to the ones of \cite{jmr}, Lemma~6.2, 
and \cite{Had00}, Lemma~2.7).
\begin{Lemma} \label{lem:approxLinf}
There is a non-decreasing function $C:(0,\infty)\rightarrow(0,\infty)$ 
and for all $T>0$, there exists 
$(u_{\parallel,j}^\lambda)_{\lambda \ge e} \subset  
L^\infty((0,T)\times\R^3)$ such that for all $\lambda \ge e$, 
\begin{eqnarray} 
\label{eq:approxLinf1}
& \|u_{\parallel,j}^\lambda (t) \|_{L^\infty ((0,T) \times  \R^3)} 
\leq C(T) \ln\lambda, \quad \mbox{ and } \quad  
\| \big((Id-P_j)u_j - u_{\parallel,j}^\lambda \big)(t)\|_{L^2 (\R^3)} 
\leq C(T) / \lambda, \quad \mbox{ for all } t  \in[0,T] , \\
\label{eq:approxLinf2}
& \| S^\lambda P_j u_j \|_{L^2 ((0,T),L^\infty ( \R^3))} 
\leq C(T) \sqrt{\ln\lambda} , \quad \mbox{ and } \quad    
\| \big((Id- S^\lambda  ) P_j u_j  (t) \|_{L^2 (\R^3)} 
\leq C(T) / \lambda  ,  \quad \mbox{ for all } t  \in[0,T] .
\end{eqnarray}
\end{Lemma}
Let us admit for a while Lemma~\ref{lem:approxLinf} 
in order to finish the proof of Theorem~\ref{unique}. 
Fix $T>0$ and consider $t\in(0,T)$. 
In the second term on the right-hand side of \eqref{brazil}, 
decompose $u_j = u_{\parallel,j}^\lambda 
+ \left( (Id-P_j)u_j - u_{\parallel,j}^\lambda \right) 
+ S^\lambda P_j u_j + (Id- S^\lambda  ) P_j u_j$. 
We infer from Lemma~\ref{lem:approxLinf} that 
\begin{equation}
\| \delta F \|_{L^2 (  \R^3)} 
\leq C_F \left( \| v_{init} \|_{L^\infty (  \R^3)} e^{KT} \right) 
\left( \| \delta U \|_{L^2 (\R^3)} + \left( C(T) \ln \lambda 
+ \| S^\lambda P_j u_j \|_{L^\infty (\R^3)} \right) 
\| \delta v \|_{L^2 (\R^3)}  
+ 2 \frac{C(T)}{\lambda} \| \delta v \|_{L^\infty (\R^3)} \right).
\end{equation}
The energy estimate, together with Gronwall's Lemma, gives 
$$
\|\delta U(t)\|_{L^2  (  \R^3)} \leq 2 C_F C(T) \frac{t}{\lambda}  
\exp \left( C_F \int_0^t (1+C(T)\ln \lambda + 
\| S^\lambda P_j u_j (t') \|_{L^\infty (\R^3)}) dt' \right) .
$$
Since $\int_0^t (1+C(T)\ln \lambda + 
\| S^\lambda P_j u_j (t') \|_{L^\infty (\R^3)}) dt'
\leq \overline C(T) \ln\lambda$ 
with $\overline C(T) \tendlorsque{T}{0} 0$, we choose 
$T_0$ small enough (in order to have $\overline C(T_0)<1$), and let
$\lambda$ go to infinity. This shows that $\delta U(t)$ vanishes on
$[0,T_0]$. Repeat this procedure on intervals of size $T_0$ to get 
the desired conclusion. 
\begin{proof}[Proof of Lemma \ref{lem:approxLinf}]
We define $u_{\parallel,j}^\lambda$ by setting 
$u_{\parallel,j}^\lambda(t,x) := (Id-P_j) u_j (t,x) \mbox{ if } |(Id-P_j) u_j (t,x)| \leq C \ln\lambda $,  and $u_{\parallel,j}^\lambda(t,x) :=
0 \mbox{ otherwise,}$
where the constant $C$ is chosen below (independently of 
$(t,x) \in [0,T]  \times \R^3$).
Therefore, for $p \in \lbrack2,\infty)$, 
\begin{equation}
\label{debase}
\|((Id-P_j) u_j - u_{\parallel,j}^\lambda)(t)\|_{L^2}^2 
= \int_{|(Id-P_j) u_j|\geqslant C\ln\lambda}|(Id-P_j) u_j (t)|^2 dx 
\leq (C\ln\lambda)^{2-p} \|(Id-P_j) u_j (t)\|_{L^p}^p .
\end{equation}
Now, according to Lemma~\ref{PonLp}, the projection  $Id-P$ 
acts continously in any $L^p$ with a norm less than $C_0 p$. 
Furthermore, we have
$$
\|v_{init}\|_{L^p}^p \le \|v_{init}\|_{L^\infty}^p |\Omega|,
$$
so that $\left( \|v_{init}\|_{L^p} \right)_{1\le p\le\infty}$ 
is bounded. Thus, using Equation~\eqref{divgen1} and the bound 
from Theorem~\ref{th:weak} \eqref{estimweak3}, we infer 
from \eqref{debase} that
\begin{equation*}
\begin{split}
\|((Id-P_j) u_j - u_{\parallel,j}^\lambda)(t)\|_{L^2}^2 
& \leq \frac{(C_0~p~e^{KT}~\|v_{init}\|_{L^p})^p}
{(C\ln\lambda)^{p-2}} 
= (C \ln\lambda)^2 \lambda^{2 \ln \left( 2 \frac{C_0}{C}
e^{KT} \sup_{1\le q\le\infty}\|v_{init}\|_{L^q} \right) },
\end{split}
\end{equation*}
choosing  $p=2 \ln \lambda$. With $C$ big enough, we obtain \eqref{eq:approxLinf1} 
($C(T)=2C_0e^{KT+1}\sup_{1\le q\le\infty}\|v_{init}\|_{L^q}$ 
is suitable).

We are now concerned with the first inequality in 
\eqref{eq:approxLinf2}. Coming back to \eqref{selec1}-\eqref{selec2},  
the idea is to use the Strichartz estimate \eqref{StrichartzEsti}. 
However, since we are not able to bound $\partial_x f$ in 
$ L^\infty_{loc} ( (0,\infty), L^2 (\R^3 ))$, we cannot apply  
\eqref{StrichartzEsti} directly. To overcome this difficulty 
we introduce some potential vectors. 
Since $P u_{init}  \in H^1  (\R^3)$ (respectively since 
$f =(f_1 ,f_2 ) \in W^{1,\infty}_{loc}( (0,\infty), L^2 (\R^3 ))$ 
(cf. \eqref{ctoucon} and \eqref{dtfesti}) and $Pf=f$), for $i=1,2$, 
there exists 
$\phi_{init} := (\phi_{init,1},\phi_{init,2}) \in H^2(\R^3)$ 
(resp.  $\psi := (\psi_1 , \psi_2) \in 
W^{1,\infty}_{loc} ( (0,\infty), H^1 (\R^3 ))$) such that 
\begin{equation*} 
\Div(\kappa_{3-i} \phi_{init,i} )=0 , \quad 
\curl( \phi_{init,i}  )= \kappa_{i} u_{init,i} \quad 
\text{ (resp.  } \Div(\kappa_{3-i} \psi_i )=0 , \quad  
\curl( \psi_i  )= \kappa_{i} f_i ).
\end{equation*}
We consider the operator ${\check B}$ defined by 
\begin{eqnarray*} 
{\check B} (\phi_1 ,\phi_2 ) = 
(\kappa_2^{-1}  \curl \phi_2 ,  -\kappa_1^{-1}  \curl \phi_1 ) 
\quad \text{ for }  \phi:=(\phi_1 ,\phi_2 ) \in 
D({\check B} ) :=  D(B) = H_{\curl} \times  H_{\curl} .
\end{eqnarray*}
The operator ${\check B}$ is simply deduced from $B$ by switching 
$\kappa_1$ and $\kappa_2$. It therefore shares the same properties 
and estimates. In addition it satisfies the identity:
\begin{eqnarray}
 \label{tuvacommuter} 
\begin{bmatrix}
  \kappa_1^{-1} \curl & 0
  \\ 0 &  \kappa_2^{-1} \curl 
\end{bmatrix}  
{\check B} = B 
\begin{bmatrix}
  \kappa_1^{-1} \curl & 0
  \\ 0 &  \kappa_2^{-1} \curl 
\end{bmatrix}  .
\end{eqnarray}
Let $\phi := (\phi_1,\phi_2)$ be the solution (for $x \in \R^3$) of 
\begin{eqnarray}
 \label{potentielvecteur} 
( \dt + {\check B} ) \phi = \psi , \quad 
\text{ with  }
\phi|_{t=0} =  \phi_{init} .
\end{eqnarray}
Using the identity \eqref{tuvacommuter} and 
Lemma \ref{keskonepropre} we obtain 
\begin{eqnarray}
 \label{noncarac}
\curl( \phi_{i}  )= \kappa_{i} u_{i} , \quad 
\text{ for } i=1,2.
\end{eqnarray}
Now observe that $ \dt \phi$ verifies 
$( \dt + {\check B} ) \dt \phi = \dt \psi$ ,
 with  $\dt \phi |_{t=0} = \psi |_{t=0}  -  {\check B}  \phi_{init} $.
Applying the Strichartz estimate \eqref{StrichartzEsti}
we obtain 
\begin{eqnarray} 
 \label{StrichartzEstiPV}
\|S^\lambda \dt \phi  \|_{L^2((0,T),L^\infty(\R^3))} \leq C(T)
\sqrt{\ln(1+\lambda T)} \Big( \| \partial_{x} \dt \phi |_{t=0}  \|_{L^2(\R^3)} +
\|  \dt \partial_{x} \psi   \|_{L^1((0,T),L^2(\R^3))} \Big).
\end{eqnarray}
From the definitions of  $\phi$ and $\psi$ and  from the estimates \eqref{ctoucon} and \eqref{dtfesti} of the previous sections  we get :
\begin{eqnarray} 
 \| \partial_{x} \dt \phi |_{t=0}  \|_{L^2(\R^3)} +
 \|  \dt \partial_{x} \psi   \|_{L^1((0,T),L^2(\R^3))} \leq C(T)  \|P  u_{init}  \|_{H^1 (\R^3)} .
 \end{eqnarray}
Using now \eqref{potentielvecteur}, observing that $f_{3-j}$ and therefore $\psi_{3-j}$ vanish because of Assumption~\ref{decouplage}, 
and using \eqref{noncarac}, we obtain the first inequality in 
\eqref{eq:approxLinf2}. 
The second one follows by applying Bernstein lemma.
\end{proof}
%


\section{Generic uniqueness: proof of Theorem~\ref{generic}}
\label{proofgeneric}

We apply the following general result of generic uniqueness for evolution equations by Saint-Raymond.
\begin{Theorem}[Saint-Raymond \cite{LSR}, Theorem $1$]
\label{LSR}
Let $\mathcal{E}_{init}$ be a topological space and $\mathcal{E}$ a metric space.
Let $(S)$ be an evolution equation admitting a solution in $\mathcal{E}$  for any initial data in $\mathcal{E}_{init}$.
Consider the following hypotheses.
\begin{enumerate}[(H1)]
\item \label{H1} For any initial data $U_{init} \in 
\mathcal{E}_{init}$, for any  $(U_{init}^ \eps )_\eps $ 
tending to  $U_{init}$ in $\mathcal{E}_{init}$, for any  
$(U^\eps )_\eps$ in $\mathcal{E}$ respective solutions to $(S)$ 
with $U^\eps_{init}$ as initial data,  
\begin{enumerate}[(i)]
\item \label{sous1} there exists a limit point of $(U^\eps )_\eps$ 
in $\mathcal{E}$;
\item \label{sous2} any limit point of $(U^\eps )_\eps$ in 
$\mathcal{E}$ is solution to $(S)$ with $U_{init}$ as initial data.
\end{enumerate}
\item \label{H2} There exists $\mathcal{D}$, dense subset of 
$\mathcal{E}$, such that for any   $ U_{init} $   in  $\mathcal{D}$, 
there exists only one solution to $(S)$  in $\mathcal{E}$ with   
$ U_{init} $   as initial data.
\end{enumerate}
Under these two hypotheses, there exists a $G_{\delta}$ dense  
$\widetilde{\mathcal{E}_{init}}$ of $\mathcal{E}_{init}$ such that 
for any $U_{init} \in  \widetilde{\mathcal{E}_{init}}$, there exists 
only one solution to $(S)$   in $\mathcal{E}$ with   $ U_{init} $   
as initial data.
 \end{Theorem}

Recall that we denote by $\tau_s$ and  $\tau_w$ respectively 
the strong and weak topologies of $L^2(\R^3,\R^6 )$,  
and by $\tilde{\tau}_s$ the strong topology  of $L^2(\Omega,\R^d)$. 
We consider the product topology ${\bf \tau}_{ss} $ 
(resp. ${\bf \tau}_{ws} $) on  ${\bf L}^2$ 
obtained from $\tau_s$ (resp. $\tau_w$) and $\tilde{\tau}_s$.

For any $C_{init} > 0$, consider 
$$
\mathcal{E}_{init} := \{ U_{init} \in \Ldiv 
\mid \| v_{init} \|_{L^\infty(\Omega)} \leq C_{init} \},
$$
endowed with the  topology ${\bf \tau}_{ss} $ 
(resp. ${\bf \tau}_{ws}$) inherited from ${\bf L^2  }$, and 
$$
\mathcal{E} := 
\{ U \in C([0,\infty),{\bf L}^2), \text{ satisfying \eqref{nunu} 
and the estimates \eqref{estimweak1}, \eqref{estimweak2}, 
\eqref{estimweak3} of Theorem~\ref{th:weak}} \}, 
$$ 
endowed with the strong topology (resp. the weak $*$ topology 
relative to ${\bf \tau}_{ws} $) of $C([0,\infty),  {\bf L}^2 )$. 
Hypothesis (H\ref{H1}) is a direct consequence of the stability 
property stated in Theorem \ref{th:weak} 
(resp. Proposition \ref{weakprinciple}). 
Now, set
$$
\mathcal{D} := \{ U_{init} \in \Ldiv   \text{ with }  
\curl  u_{init,i}  \in L^2 (\R^3), \text{ for } i=1,2 \},
$$
which is dense in  $\mathcal{E}_{init}$ for the topology  
${\bf \tau}_{ss} $  inherited from ${\bf L^2  }$. 
Moreover Theorem~\ref{unique} yields that Hypothesis (H\ref{H2}) 
is satisfied. 
We can therefore apply Theorem \ref{LSR}, what proves  
Theorem~\ref{generic}.


\section{Quasi-stationary limits: proof of Theorem~\ref{th:weakquasi},  
Proposition~\ref{prop:uniqquasi} and Theorem~\ref{th:quasi}}
\label{proofth:quasi}

\begin{proof}[Proof of Theorem~\ref{th:weakquasi}.] 
We first observe that the bounds \eqref{estimweak1}, 
\eqref{estimweak2} given by Theorem~\ref{th:weak} 
for $U^\eta $  are uniform in $\eta \in (0,1)$. 
Therefore, up to a subsequence,  $U^\eta $  converges 
to $U:= (u,v)$ in  
$W^{1,\infty}_{loc}((0,\infty ),L^2(\Omega)) \text{ weak }*$ 
and $v^\eta$ converges to $v$ in  
$L^\infty_{loc} ((0,\infty ),L^\infty (\Omega)) \text{ weak } *$. 
In addition, there holds  $B u^\eta  =  \dt D^\eta $, 
with $D^\eta := - \eta (u^\eta - (\kappa^{-1} \cdot l) 
\overline{v^\eta} )$. 
Passing to the limit already yields that $U$ satisfies  the 
(linear) equations  \eqref{divgen1} and $Bu= 0$. 
Using Proposition \ref{weakprinciple}, we also get that  
$U$ satisfies \eqref{MBgen2}, which means that $v$ solves 
\eqref{eq:limquasiv}. 
\end{proof}

\begin{proof}[Proof of Proposition~\ref{prop:uniqquasi}.] 
The proof is very similar to the uniqueness proof 
in Theorem~\ref{unique} \eqref{Unique2}: it relies on some 
$L^\infty$ approximation of $(Id-P)(\kappa^{-1}\cdot l)v$. 
Consider $v_1, v_2 \in C([0,\infty),L^2(\Omega)) \cap 
L^\infty_{loc}((0,\infty ),L^\infty (\Omega))$, solutions to 
\eqref{eq:limquasiv} with the same initial data $v_{init}$, and 
define $\delta v := v_1-v_2$. Fix $T>0$. From the properties 
\eqref{affine}-\eqref{maj2} of $F$, we get the pointwise estimate 
\begin{equation} \label{eq:estimuniqquasi}
\dt (|\delta v|^2) 
\le C_F \left( (|(Id-P)(\kappa^{-1}\cdot l)v_1| + 1) |\delta v|^2 
+ |(Id-P)(\kappa^{-1}\cdot l)\delta v| \, |\delta v|
\right) 
\quad \text{on} \quad \lbrack0,T\rbrack\times\Omega,
\end{equation}
for some constant 
$C_F=C_F\left( \|v_{init}e^{KT}\|_{L^\infty(\Omega)} \right)$.  
Now, defining for $M>0$
$$
w^M_\parallel := {\bf 1}_{|(Id-P)(\kappa^{-1}\cdot l)v_1|\le M} 
(Id-P)(\kappa^{-1}\cdot l)v_1 ,
$$
we get from Lemma~\ref{lem:approxLinf} that there is $C(T)>0$ 
such that
$$
\forall M\ge1, \quad 
\| w^M_\parallel \|_{L^\infty((0,T)\times\Omega)} \le M , \quad 
\| (Id-P)(\kappa^{-1}\cdot l)v_1 - w^M_\parallel
\|_{L^\infty((0,T),L^2(\Omega))} \le C(T)e^{-M/C(T)} .
$$
Integrating \eqref{eq:estimuniqquasi} over $\Omega$, using the 
Cauchy-Schwarz inequality and increasing the constant $C$ 
(which is still independent of $M$), we obtain
$$
\dt \left( \| \delta v \|_{L^2(\Omega)}^2 \right) \le 
C_F \left( (M+1)\| \delta v \|_{L^2(\Omega)}^2 
+ C(T)e^{-M/C(T)} \right).
$$
Then, Gronwall's lemma yields 
$$
\forall t\in\lbrack0,T\rbrack, \quad 
\| \delta v(t) \|_{L^2(\Omega)}^2 \le 
\frac{C(T)}{M+1} e^{C_F(M+1)T-M/C(T)}.
$$
Now, choose $T$ so small that $C_FMT-M/C(T)<0$ (which is possible, 
since $C$ is a non-decreasing function of $T$), and let $M$ go to 
infinity. This shows that $\delta v$ vanishes on $[0,T]$. Repeating 
the argument on successive time intervals yields $v_1=v_2$.  
\end{proof} 

\begin{proof}[Proof of Theorem~\ref{th:quasi}.] 
For each $\eta\in(0,1)$, consider a solution $U^\eta$ (given by 
Theorem~\ref{th:weak}) to \eqref{divgen1}, \eqref{MBgen2} and 
\eqref{MBgeneta}. Convergence of $v^\eta$ (and $(Id - P)u^\eta$) 
is obtained as in the proof of Theorem~\ref{th:weakquasi} above. 
Now, drop the index $\eta$ for simplicity. 
Then, symmetrizing the system by the change of dependent variables
$$
\widetilde{u_i} = \kappa_i^{1/2} u_i,
$$
we get in the distributional sense:
\begin{equation*}
\text{for } i=1,2, \quad  \eta  \dt  \widetilde{u_i} 
+ (-1)^{3-i}  R_{3-i}  \widetilde{u}_{3-i} 
=  \eta  \kappa_i^{-1/2} \cdot l_i \dt\bar{v} ,
\end{equation*}
and therefore, applying $ \dt$ and combining, 
\begin{equation} \label{eq:ordre2eta}
\text{for } i=1,2, \quad 
\eta^2 \dt^2 \widetilde{u_i} + R_i^\star R_i \widetilde{u_i} =  
(-1)^i \eta R_{3-i} \kappa_{3-i}^{-1/2} l_{3-i} \dt\bar{v} 
+ \eta^2 \kappa_i^{-1/2} l_i \dt^2\bar{v},
\end{equation}
where we have set 
$$
i=1,2, \quad R_i := \kappa_{3-i}^{-1/2} \curl \kappa_i^{-1/2} 
(=R_{3-i}^\star, \text{ for the duality in } L^2(\R^3,dx)). 
$$
System \eqref{eq:ordre2eta} shall be understood as a system of wave 
equations for the ``divergence free'' parts $\pi_i\widetilde{u_i}$, 
when 
$$
\text{for } i=1,2, \quad \pi_i := \kappa_i^{1/2} P_i \kappa_i^{-1/2}. 
$$
Then, $\pi = (\pi_1,\pi_2)$ is an orthogonal projector in the space 
$L^2(\R^3,dx) \times L^2(\R^3,dx)$. Furthermore, 
from the description of $\text{ran }P_i$ and $\text{ran }(1-P_i)$ 
in \eqref{eq:ranP,kerP}, we deduce that 
$$
\text{for }i=1,2, \quad R_i \pi_i = R_i \quad \text{(and} \quad 
\pi_i R_i^\star = R_i^\star \quad \text{by transposition)}. 
$$
Thus, we have finally:
\begin{equation} \label{eq:ondeseta}
\text{for } i=1,2, \quad 
\eta^2 \dt^2 \pi_i\widetilde{u_i} - Q_i \pi_i\widetilde{u_i} =  
(-1)^i \eta R_{3-i} \kappa_{3-i}^{-1/2} l_{3-i} \dt\bar{v} 
+ \eta^2 \pi_i \kappa_i^{-1/2} l_i \dt^2\bar{v},
\end{equation}
with 
$$
\text{for }i=1,2, \quad Q_i := - R_i^\star R_i 
+ \kappa_i^{1/2} \nabla \left( \kappa_i^{-2} \kappa_{3-i}^{-1} 
\Div (\kappa_i^{1/2}\cdot) \right). 
$$
From \cite[Lemma 3.10]{starynkevitchPhD}, we know that for $i=1,2$, 
the differential second-order operator $(-Q_i)$ is a self-adjoint, 
positive, and elliptic. Thus, with 
$v \in L^\infty_{loc}((0,\infty),L^\infty(\Omega)) \cap 
W^{1,\infty}_{loc}((0,\infty),L^2(\Omega))$ given, and for given 
initial data,  
\begin{equation*}
\pi_i \widetilde{u_i}_{|_{t=0}} = \pi_i \kappa_i^{1/2} u_{init,i} 
\text{ and } \eta ( \dt \pi_i \widetilde{u_i})_{|_{t=0}} 
= \eta \pi _i \kappa_i^{-1/2} l_i F(\overline{v_{init}},u_{init}) 
+ (-1)^i R_{3-i} \kappa_{3-i}^{1/2} u_{init,3-i} ,
\end{equation*}
the solution $(\pi_1\widetilde{u_1}, \pi_2\widetilde{u_2})$ 
to the linear wave equation system \eqref{eq:ondeseta} is 
uniquely determined. We recover it \emph{via} vector potentials: 
defining 
\begin{equation}
 \label{dumaspabo} 
 \left\{
\begin{split}
\text{for } i=1,2, \quad \eta^2 \dt^2 \phi_i - Q_i \phi_i 
& = \eta \pi_i \kappa_i^{-1/2} l_i \dt \bar{v} , \\
\phi_{i |_{t=0}} & = 0 , \\ 
\eta \dt \phi_{i |_{t=0}} & = \pi_i \widetilde{u_i}_{|_{t=0}},
\end{split}
\right.
\end{equation}
we have
\begin{equation} \label{eq:uitildepotentiel}
\text{for } i=1,2, \quad \pi_i \widetilde{u_i} 
= \eta \dt \phi_i + (-1)^i R_{3-i} \phi_{3-i}. 
\end{equation}
The problem  \eqref{dumaspabo} also determines uniquely the 
vector potentials $\phi_i$.
Since $\pi_i \phi_i $ also satisfies the problem  \eqref{dumaspabo} we infer that $\pi_i \phi_i =\phi_i$.
Furthermore, \cite[Lemma 3.10]{starynkevitchPhD} ensures that 
for $i=1,2$, $Q_i$ does not admit 0 as a resonance. One then 
needs to assume the following:
\begin{equation} \label{hyp:nontrapping}
\text{for } i=1,2, \, Q_i \text{ is non-trapping}.
\end{equation}
This is enough to apply
\begin{Theorem}[Starynkevitch \cite{starynkevitchPhD}, Theorem 3.2]
Let $Q$ be a non-trapping, ($L^2$-)self-adjoint, negative, 
and elliptic differential second-order operator, for which 0 is 
not a resonance. Let $s>1/2$, $\gamma \in (-3/2,1/2)$ and $R>0$. 
Then, there exists $C\ge0$ such that: for all 
$(u_0,u_1) \in \dot H_Q^{\gamma+1}(\R^3) \times \dot H_Q^\gamma(\R^3)$, 
and $f$ such that 
$\langle x \rangle^s (-Q)^{\gamma/2} f \in L^2((0,\infty)\times\R^3)$, 
the solution $u$ to 
$$
\dt^2 u - Q u = f \text{ on } (0,\infty)\times\R^3, \quad 
\text{ with } u_{|_{t=0}} = u_0 , \quad \dt u_{|_{t=0}} = u_1 , 
$$
satisfies 
$$
\| (u,\dt u) \|_{L^2((0,\infty), \dot H_Q^{\gamma+1}(B_R) \times 
\dot H_Q^\gamma(B_R))} \le 
C \left( \| u_0 \|_{\dot H_Q^{\gamma+1}(\R^3)} + 
\| u_1 \|_{\dot H_Q^\gamma(\R^3)} + 
\| \langle x \rangle^s (-Q)^{\gamma/2} f \|_{L^2((0,\infty)\times\R^3)}
\right) .
$$
\end{Theorem}
For all $\mu\in\R$, the space $\dot H_Q^\mu(\R^3)$ is defined by 
the norm 
$$
\| v \|_{\dot H_Q^\mu(\R^3)} = \| (-Q)^{\mu/2} v \|_{L^2(\R^3)} .
$$
We apply the result above to $\phi_i(\eta t,x)$, 
whith $\gamma=0$ and $s=1$. This leads to:  
\begin{equation*}
\begin{split}
\forall R>0, \, \exists C_R>0, \quad 
\| (\phi_i,\eta\dt\phi_i) \|_{L^2((0,\infty),
\dot{H}^1(B_R)\times L^2(B_R))} 
\le C_R 
& \big( \eta^{1/2}  \| \pi_i \kappa_i^{1/2} u_{init,i} \|_{L^2(\R^3)} \\
& + \eta \| \langle x \rangle \pi_i \kappa_i^{-1/2} l_i \dt\bar{v} 
\|_{L^2((0,\infty),L^2(\R^3))} \big).
\end{split}
\end{equation*}
The right-hand side is controlled thanks to 
\begin{Lemma}[Starynkevitch \cite{starynkevitchPhD}, Lemma 3.11]
For all $R>0$, there exists $C_R>0$ such that,  
if $m \in L^2(\R^3)$ and ${\rm supp} (m) \subset B_R$, then 
$$
\text{for } \, i=1,2,
\quad |\pi_i m (x)| \le C_R \langle x \rangle^{-3} \| m \|_{L^2(\R^3)} 
\quad \text{for } \, a.e. \, x\in\R^3.
$$
\end{Lemma}
Since $\pi_i\phi_i=\phi_i$, by the usual $TT^\star$ argument, 
$\| R_i\phi_i \|_{L^2(B_R)} \le \| \phi_i \|_{\dot{H}^1(B_R)}$, 
and we deduce from \eqref{eq:uitildepotentiel}: 
$$
\text{for } \, i=1,2, \quad
\pi_i \widetilde{u_i} = O(\sqrt{\eta}) \quad 
\text{ in } L^2((0,\infty),L^2_{loc}(\R^3)),
$$
which yields the convergence of $P u^\eta$ to zero in $L^2((0,\infty),L^2_{loc}(\R^3))$.
\end{proof}


\end{document}